\documentclass[a4paper,11pt]{article}
\usepackage{array}
\usepackage{amsmath,amscd,amssymb,amsthm} 
\usepackage{latexsym}
\usepackage{epsfig}
\usepackage{enumerate}
\usepackage{xcolor}

\theoremstyle{plain}
\newtheorem{lemma}{Lemma}[section]

\newtheorem{proposition}[lemma]{Proposition}
\newtheorem{theorem}[lemma]{Theorem}

\newtheorem{definition}[lemma]{Definition}
\theoremstyle{definition}

\theoremstyle{remark}
\newtheorem{remark}[lemma]{Remark}
\newcommand{\rk}{\mathrm{rk}}
\newcommand{\pp}{{\mathbb P}}

\newcommand{\PP}{\mathbb P}
\newcommand{\ZZ}{\mathbb Z}
\newcommand{\cB}{\mathcal B}
\newcommand{\cE}{\mathcal E}
\newcommand{\cF}{\mathcal F}
\newcommand{\cL}{\mathcal L}
\newcommand{\cO}{\mathcal O}

\newcommand{\cC}{\mathcal C}
\newcommand{\cG}{\mathcal G}
\newcommand{\cA}{\mathcal A}
\newcommand{\cV}{\mathcal V}
\newcommand{\cM}{\mathcal M}
\newcommand{\cI}{{\mathcal I}}

\newcommand{\Pic}{\mathop{\mathrm {Pic}}\nolimits}
\newcommand{\Proj}{\mathop{\null\mathrm {Proj}}\nolimits}

\newcommand{\Sym}{\mathop{\mathrm {Sym}}\nolimits}
\newcommand{\Hom}{\mathop{\mathrm {Hom}}\nolimits}

\newcommand {\sHom}{\mathcal{H}\kern -0.25ex{\mathit om}}
\newcommand {\sExt}{\mathcal{E}\kern -0.25ex{\mathit xt}}
\def\red#1{\textcolor{red}{#1}}
\def\H#1{\mathrm{H}^{#1}}
\def\h#1{\mathrm{h}^{#1}}

\makeatletter
\renewcommand\@makefntext[1]{%
\setlength\parindent{1em}%
\noindent
\mbox{\@thefnmark}{#1}}
\makeatother
\begin{document}
\title{H-Instanton Bundles on Three-Dimensional Smooth Toric Varieties with Picard Number Two}
\author{Ozhan Genc and Francesco Malaspina}
\maketitle
\def\thefootnote{\null}

\begin{abstract}\noindent
We study $H$-instanton bundles on the infinite family of smooth three-dimensional varieties $X_e=\mathbb{P}(\cO_{\PP^2}\oplus\cO_{\PP^2}(e))$, for $e \geq 0$. We provide two distinct monadic descriptions of $H$-instanton bundles on $X_e$, generalizing the classical monads on $\mathbb P^3$. We then characterize $H$-instanton bundles with second Chern class supported in a single degree, and investigate their existence and moduli spaces. Finally, for $e\leq 3$, we prove the existence of  $H$-instanton bundles for all admissible second Chern classes. These results extend previous constructions on specific cases and contribute to the study of instanton bundles on threefolds with higher Picard number.
\end{abstract}

\section*{Introduction}\label{sect:intro}

Mathematical instanton bundles on $\mathbb P^3$ arose as the algebraic counterpart to the
(anti)self-dual connections which are solutions to the Yang-Mills equations on $S^4$
. The
link between these two apparently distant objects was achieved through twistor theory,
as it was developed by R. Penrose. Indeed $\mathbb P^3$ is the twistor space of $S^4$. Prompted by
this fact, mathematical instanton bundles on $\mathbb P^3$ have been the object of study for the
last three decades, becoming one of the milestones in algebraic geometry. An instanton
bundle on $\mathbb P^3$ is a stable rank two bundle $E$ with $c_1(E)=0$ and satisfying the so called
instantonic condition
$$h^1(\PP^3,E(-2))=0.$$
In \cite{Fa} (see also [28] in the case of index two)  is given the following definition of instanton bundle on Fano
threefold X with Picard number one: let us write the canonical divisor of $(X, H_X)$ as
$\omega_X =-(2qX+r_X )H_X$ where $q_X\geq 0$ and $0\leq r_X\leq 1$, then an instanton bundle E on X
is a $\mu$-stable rank two vector bundle with $c_1(E) =-r_X$ (hence normalized) and
$$H^1(X, E(-q_X H_X ))=0.$$
So $G=E(-q_X H_X )$ satisfies $G\cong G^\vee\otimes\omega_X$ and by Serre duality $H^1(X, G)=H^2(X, G)=0$.
Apart from this definition there have been several attempts to generalize instanton
bundles to Fano varieties of higher Picard number. In \cite{MMP} the authors gave the following definition of instanton bundle over the Flag variety $F(0,1,2)$:
a rank  two vector bundle $E$ on the Fano threefold $F(0,1,2)$ is an instanton bundle of charge $k$ if the following properties hold:
\begin{itemize}
\item $c_1(E)=0, \ c_2(E)=kh_1h_2$;
\item $h^0(E)=0$ and $E$ is $\mu$-semistable;
\item $h^1(X, E(-H_X))=0$;
\end{itemize}
This definition has been adapted in \cite{AM}, \cite{CCGM} and \cite{CG} to other examples of Fano threefolds with Picard number higher than one. One can observe that in all the cases, by
increasing the degree of the second Chern class, the intermediate cohomology modules
have more generators and become more complicated. Recall that a coherent sheaf E is
aCM (arithmetically Cohen-Macauly) if all the intermediate cohomology groups vanish
or equivalently if the module E of global sections of E is a maximal Cohen-Macaulay
module. So, as the charge of instanton bundles grows, instanton bundles become further
from being aCM. Among aCM sheaves, a special role is played by Ulrich sheaves. These
shaves are characterized by the linearity of the minimal graded free resolution over the
polynomial ring of their module of global sections. Ulrich sheaves, originally studied for
computing Chow forms, conjecturally exist over any variety (we refer to \cite{ES}).

In \cite{AM2} a notion of instanton bundles related to the previous ones is introduced on
any polarized projective threefold $(X, H_X)$.  Let $K_X$ be the
canonical divisor,
A rank two vector bundle $E$ on $X$ is an \emph{$H$-instanton bundle} of charge $k$ if the following properties hold:
\begin{itemize}
\item $c_1(E)=2H_X+K_X$ and $\deg(c_2(E))=k$;
\item $h^0(X,E)=0$ and $E$ is $\mu$-semistable;
\item $h^1(X,E(-H_X))=0$;
\end{itemize}

In \cite{AM2} has been studied, H-instanton
bundles on  three-dimensional rational normal scrolls since this is a family of threefolds with infinite
different canonical divisors $K_X$ and with a nice description of the derived category of
coherent sheaves.

In this paper we deal with the infinite family $X_e=\mathbb{P}(\cO_{\PP^2}\oplus\cO_{\PP^2}(e))$, namely the remaining three-dimesional smooth toric varieties with Picard number two (cfr \cite{Kl}). The case of $e=0$ has already been studied in \cite{AM2} (see also \cite{CG} and \cite{GJ}), the case of $e=1$ in \cite{CCGM} and the case of $e=2$ in \cite{GJ}. As first results, in Section 3 we give two different monadic descriptions
of each H-instanton bundle (which can be considered as the analog of the monads on
$\PP^3$ and we characterize the H-instantons $E$ such that $c_2(E)=\beta f^2$ (i.e. the second
Chern class is concentrated in a single degree) showing that they are always realized as cohomology of a  pullback of a linear monad.
After this, in Section 4 we deal with the existence of H-instanton bundles with second Chern classes $\alpha \xi f$ and the
description of their moduli spaces. Finally, in Section 5, we show the existence of H-instanton bundles for any admissible second Chern classes when $e \leq 3$.

\section{Preliminaries}\label{sect:scrolls}

\begin{definition}
A coherent sheaf $\cE$ on a projective variety $X_e$ with a fixed ample line bundle $\cO_X(1)$ is called {\it arithmetically Cohen-Macaulay} (for short, aCM) if it is locally Cohen-Macaulay and $H^i(\cE(t))=0$ for all $t\in \ZZ$ and $i=1, \ldots, \dim (X)-1$.
\end{definition}



\begin{definition}
For an {\it initialized} coherent sheaf $\cE$ on $X$, i.e. $h^0(\cE(-1))=0$ but $h^0(\cE)\ne 0$, we say that $\cE$ is an {\it Ulrich sheaf} if it is aCM and $h^0(\cE)=\deg (X)\mathrm{rank}(\cE)$.
\end{definition}

Given a smooth projective variety $X$, let $D^b(X)$ be the the bounded derived category of coherent sheaves on $X$. An object $E \in D^b(X)$ is called {\it exceptional} if $Ext^\bullet(E,E) = \mathbb C$.
A set of exceptional objects $\langle E_0, \ldots, E_n\rangle$ is called an {\it exceptional collection} if $Ext^\bullet(E_i,E_j) = 0$ for $i > j$. An exceptional collection is said to be {\it full} when $Ext^\bullet(E_i,A) = 0$ for all $i$ implies $A = 0$, or equivalently when $Ext^\bullet(A, E_i) = 0$ does the same.

\begin{definition}\label{def:mutation}
Let $E$ be an exceptional object in $D^b(X)$.
Then there are functors $\mathbb L_{E}$ and $\mathbb R_{E}$ fitting in distinguished triangles
$$
\mathbb L_{E}(T) 		\to	 Ext^\bullet(E,T) \otimes E 	\to	 T 		 \to	 \mathbb L_{E}(T)[1]
$$
$$
\mathbb R_{E}(T)[-1]	 \to 	 T 		 \to	 Ext^\bullet(T,E)^* \otimes E	 \to	 \mathbb R_{E}(T)	
$$
The functors $\mathbb L_{E}$ and $\mathbb R_{E}$ are called respectively the \emph{left} and \emph{right mutation functor}.
\end{definition}


The collections given by
\begin{align*}
E_i^{\vee} &= \mathbb L_{E_0} \mathbb L_{E_1} \dots \mathbb L_{E_{n-i-1}} E_{n-i};\\
^\vee E_i &= \mathbb R_{E_n} \mathbb R_{E_{n-1}} \dots \mathbb R_{E_{n-i+1}} E_{n-i},
\end{align*}
are again full and exceptional and are called the \emph{right} and \emph{left dual} collections. The dual collections are characterized by the following property; see \cite[Section 2.6]{GO}.
\begin{equation}\label{eq:dual characterization}
Ext^k(^\vee E_i, E_j) = Ext^k(E_i, E_j^\vee) = \left\{
\begin{array}{cc}
\mathbb C & \textrm{\quad if $i+j = n$ and $i = k $} \\
0 & \textrm{\quad otherwise}
\end{array}
\right.
\end{equation}

\begin{theorem}[Beilinson spectral sequence]\label{thm:Beilinson}
Let $X$ be a smooth projective variety and with a full exceptional collection $\langle E_0, \ldots, E_n\rangle$ of objects for $D^b(X)$. Then for any $A$ in $D^b(X)$ there is a spectral sequence
with the $E_1$-term
\[
E_1^{p,q} =\bigoplus_{r+s=q} Ext^{n+r}(E_{n-p}, A) \otimes \mathcal H^s(E_p^\vee )
\]
which is functorial in $A$ and converges to $\mathcal H^{p+q}(A)$.
\end{theorem}

The statement and proof of Theorem \ref{thm:Beilinson} can be found both in  \cite[Corollary 3.3.2]{RU}, in \cite[Section 2.7.3]{GO} and in \cite[Theorem 2.1.14]{BO}.


Let us assume next that the full exceptional collection  $\langle E_0, \ldots, E_n\rangle$ contains only pure objects of type $E_i=\mathcal E_i^*[-k_i]$ with $\mathcal E_i$ a vector bundle for each $i$, and moreover the right dual collection $\langle E_0^\vee, \ldots, E_n^\vee\rangle$ consists of coherent sheaves. Then the Beilinson spectral sequence is much simpler since
\[
E_1^{p,q}=Ext^{n+q}(E_{n-p}, A) \otimes E_p^\vee=H^{n+q+k_{n-p}}(\mathcal E_{n-p}\otimes A)\otimes E_p^\vee.
\]

Note however that the grading in this spectral sequence applied for the projective space is slightly different from the grading of the usual Beilison spectral sequence, due to the existence of shifts by $n$ in the index $p,q$. Indeed, the $E_1$-terms of the usual spectral sequence are $H^q(A(p))\otimes \Omega^{-p}(-p)$ which are zero for positive $p$. To restore the order, one needs to change slightly the gradings of the spectral sequence from Theorem \ref{thm:Beilinson}. If we replace, in the expression
\[
E_1^{u,v} = \mathrm{Ext}^{v}(E_{-u},A) \otimes E_{n+u}^\vee=
H^{v+k_{-u}}(\mathcal E_{-u}\otimes A) \otimes \mathcal F_{-u}
\]
$u=-n+p$ and $v=n+q$ so that the fourth quadrant is mapped to the second quadrant, we obtain the following version (see \cite{AHMP}) of the Beilinson spectral sequence:

\begin{theorem}\label{use}
Let $X$ be a smooth projective variety with a full exceptional collection
$\langle E_0, \ldots, E_n\rangle$
where $E_i=\mathcal E_i^*[-k_i]$ with each $\mathcal E_i$ a vector bundle and $(k_0, \ldots, k_n)\in \ZZ^{\oplus n+1}$ such that there exists a sequence $\langle F_n=\mathcal F_n, \ldots, F_0=\mathcal F_0\rangle$ of vector bundles satisfying
\begin{equation}\label{order}
\mathrm{Ext}^k(E_i,F_j)=H^{k+k_i}( \mathcal E_i\otimes \mathcal F_j) =  \left\{
\begin{array}{cc}
\mathbb C & \textrm{\quad if $i=j=k$} \\
0 & \textrm{\quad otherwise}
\end{array}
\right.
\end{equation}
i.e. the collection $\langle F_n, \ldots, F_0\rangle$ labelled in the reverse order is the right dual collection of $\langle E_0, \ldots, E_n\rangle$.
Then for any coherent sheaf $A$ on $X$ there is a spectral sequence in the square $-n\leq p\leq 0$, $0\leq q\leq n$  with the $E_1$-term
\[
E_1^{p,q} = \mathrm{Ext}^{q}(E_{-p},A) \otimes F_{-p}=
H^{q+k_{-p}}(\mathcal E_{-p}\otimes A) \otimes \mathcal F_{-p}
\]
which is functorial in $A$ and converges to
\begin{equation}
E_{\infty}^{p,q}= \left\{
\begin{array}{cc}
A & \textrm{\quad if $p+q=0$} \\
0 & \textrm{\quad otherwise.}
\end{array}
\right.
\end{equation}
\end{theorem}

\section{First properties of $H$-instanton over $\PP(\cO_{\PP^2}\oplus\cO_{\PP^2}(e))$}

We fix a decomposable vector bundle $\cV=
\cO_{\PP^2}\oplus\cO_{\PP^2}(e)$ of rank $2$ on $\PP^2$ with $e\geq 0$. The associated projective space bundle
$X:=\PP\cV$ is by definition $\Proj(\Sym \cV)$, adopting the
notational conventions of \cite[Section~II.7]{Ha}. We have the classes $\xi$ and $f$ of $\cO_{\PP(\cV)}(1)$ and $\pi^*\cO_{\pp^2}(1)$ respectively.  Thus we have an isomorphism
$$
A(X)\cong\mathbb Z[\xi,f]/(f^3,\xi^2-e\xi f),
$$

Recall that  $\xi^3=e^2$, $\xi^2 f=e\xi f^2$ and $\xi f^2$ is the class of a point. We have $\omega_X \cong \cO_X(-2\xi+(e-3)f)$. 

The following  lemma is useful for computation.
\begin{lemma}\label{lem0}
Let $X_e$ be a scroll. Since  $\Sym^a\cV = \oplus_{j=0}^{a} \cO_X (je)$, we get
\begin{enumerate}[(i)]
\item $H^i(X, \cO_X(a\xi+bf)) \cong H^i(\PP^2, \oplus_{j=0}^{a} \cO_X (je) \otimes
  \cO_{\PP^2}(b))$ if $a\ge 0$;
\item $H^i(X, \cO_X(-\xi+bf)) =0$;
\item $H^i(X, \cO_X(a\xi+bf)) \cong H^{3-i}(\PP^2, \oplus_{j=0}^{-a-2} \cO_X (je)
  \otimes \cO_{\PP^2}(e-b-3))$ if $a<-1$.
\end{enumerate}
\end{lemma}
\begin{proof}
  See \cite[Exercise III.8.4]{Ha}.
\end{proof}

We will also make use the following exact sequences
\begin{gather}
\label{Om}
0\longrightarrow\pi^*\Omega^1_{\pp^2}\longrightarrow\cO_X(-f)^{\oplus3}\longrightarrow\cO_X\longrightarrow0,\\
\label{OmD}
0\longrightarrow\cO_X(-3f)\longrightarrow\cO_X(-2f)^{\oplus3}\longrightarrow\pi^*\Omega^1_{\pp^2}\longrightarrow0,\\
\label{Omt}
0\longrightarrow\cO_X(-3f)\longrightarrow\cO_X(-2f)^{\oplus3}\longrightarrow\cO_X(-f)^{\oplus3}\longrightarrow\cO_X\longrightarrow0,\\
\label{Rel}
0\longrightarrow\cO_X(-2\xi +ef)\longrightarrow\cO_X(-\xi)\oplus\cO_X(-\xi +ef)\longrightarrow\cO_X\longrightarrow0.
\end{gather}
The first exact sequence is the pull--back of the Euler exact sequence on $\pp^2$ via $\pi$. The second one is the pull--back of the dual of the same sequence. The third one is the relative Euler exact sequence on $X_e$.

 Thus the Riemann--Roch theorem for a rank two bundle $\cF$ on $X_e$ becomes
\begin{equation}
  \label{RRgeneral}
  \begin{aligned}
    \chi(\cF)=2&+{\frac16}(c_1(\cF)^3-3c_1(\cF)c_2(\cF)+3c_3(\cF))-\\
    &-{\frac14}(K_{X_e}c_1(\cF)^2-2K_{X_e}c_2(\cF))+{\frac1{12}}(K_{X_e}^2c_1(\cF)+c_2(\Omega_{X_e}^1) c_1(\cF))
  \end{aligned}
\end{equation}
(see \cite[Theorem A.4.1]{Ha}).

In our case $K_{X_e}=-2\xi-(3-e)f$ and $c_2(\Omega_{X_e}^1)=6\xi f+(3-3e)f^2.$

We consider $H$-instantons bundles (see \cite{AM2}) with respect to the polarization $H=\xi+f$.
On $X_e$ Definition 1.1 of \cite{AM2} takes the following form.
\begin{definition}\label{scrollinsta}
A rank two $\mu$-semistable vector bundle $\cE$ on $X_e$ is called an $H$-instanton bundle of charge $k$ if and only if
\begin{itemize}
\item $c_1(\cE)=(e-1)f$;
\item $h^0(X,\cE)=h^1(S,\cE(-H))=0$
\item $c_2(\cE)=\alpha\xi f+\beta f^2$ with $(e+1)\alpha+\beta=k$.
\end{itemize}
\end{definition}

As a consequence of Definition \ref{scrollinsta}, since $\cE^\vee\cong\cE(-(e-1)f)$, we have by Serre duality that
\begin{equation}\label{sd}
h^i(X,\cE(aH+bf))=h^{3-i}(X,\cE((-a-2)H-(b+2)f))
\end{equation}
and in particular $h^i(X,\cE(-H))=0$ for any $i$.\\
Moreover, if $\cE$ is an $H$-instanton bundle with charge $\alpha\xi^2+\beta f^2$ on $F$, then  yields:
\begin{equation}
\label{RRB}
\begin{aligned}
\chi(\cE(a\xi+b f))&=\frac{e^2}3 a^3 + e a^2b + ab^2 + (e^2 + e) a^2+b^2+ (2e+2) ab+\\
&+\big(\frac{7e^2}{6}+\frac{3e}{2}-e\alpha-\beta +1 \big)a + (e-\alpha+2)b + \\
&+ \big(\frac{e^2}{2}+\frac{e}{2}-e\alpha - \alpha -\beta +1 \big).
\end{aligned}
\end{equation}

Notice that $\cE(H)$ is Ulrich if and only if $$\chi(\cE)=-h^1(X,\cE)= \big(\frac{e^2}{2}+\frac{e}{2}-e\alpha - \alpha -\beta +1 \big) =0.$$

\begin{remark}\label{rem:exceptional divisor}
		Notice that the linear system $|\cO_{X_e}(\xi)|$ is base point free. Then it determines the morphism $f: X_e \to X_e^{'}$ which contracts to the exceptional divisor $E$ (see \cite[p. 27]{PS99}), where $X_e^{'}$ is a log Fano variety. Then we have
		\[K_{X_e} = f^* K_{X_e^{'}} + \big(\frac{3}{e}-1 \big)E
		\]
		where $f^* K_{X_e^{'}} = - \big( \frac{3}{e} +1 \big) \xi$. Then we have $E = \xi -ef$.
	\end{remark}
	
	\begin{definition}\label{defn:Earnest}
		Let $\cE$ be an $H$-instanton bundle on a projective threefold $(X,H)$. If $\h1 \big(X,\cE(-H-D)\big)=0$ for each integral smooth effective divisor $D\subseteq X$, then we say that $\cE$ is earnest.
	\end{definition}
	
	\begin{proposition}\label{prop:glob. gen. divisors}
		Let $D = a\xi + b f$ be a divisor on $X_e$. $\cO_{X_e}(D)$ is globally generated if and
		only if $a, b \geq 0$. 
	\end{proposition}
	
	\begin{proof}
		If $a, b \geq 0$ then $\cO_{X_e}(D)$ is trivially globally generated. Because both $\cO_{X_e}(\xi)$ and $\cO_{X_e}(f)$ are globally generated. 
        
        On the other hand, if $\cO_{X_e}(D)$ is globally generated then $a = D \cdot f^2$ and $b = D \cdot H \cdot E$ are non-negative.
	\end{proof}
	
	\begin{proposition}\label{prop: smooth int. divisor}
		The linear system $|a\xi + bf|$ contains a smooth integral divisor $D$ if and only if either $a, b \geq 0$ or $a =1$ and $b=-e$.
	\end{proposition}
	
	\begin{proof}
		If $a = 1$ and $b= -e$, then $D = E$ and trivially smooth and integral. If $a, b \geq 0$,
		then $\cO_{X_e}(a\xi + bf )$ is globally generated by Proposition \ref{prop:glob. gen. divisors}, and so $|a\xi + bf|$ contains a smooth integral divisor by the Bertini theorem.
		
		Conversely, assume that $\vert a\xi+bf\vert$ contains a smooth divisor. If the line bundle $\cO_{X_e} (a\xi+b f)$ is not globally generated, then we have $a\ge1$ and $-1 \ge b\ge -ea$ by Lemma \ref{lem0}.
		On the one hand, if $D\ne E$, then $D\cap E$ is necessarily a curve or it is empty, hence we should have $0\le DEH=b\le -1$, a contradiction. Since $D$ is smooth, it follows that $D=E$ which is trivially smooth and integral.
	\end{proof}

    \begin{remark}\label{ear}
		If $\cE$ is earnest, then $\h2 (\cE (-(e+1)f))$. In fact $\h1 (\cE(-h-E)) = 0$ for $D = E$ and $\h1 (\cE(-h-E)) = \h2 (\cE (-(e+1)f))$ by Serre duality.
	\end{remark}

Let us recall the notion of (semi)stability. For each sheaf $\cE$ on $X_e$ the {\sl slope} of $\cE$ with respect to $H$ is the rational number $\mu(\cE):=c_1(\cE)H^2/\rk(\cE)$ and the reduced Hilbert
polynomial $p_{\cE}(t)$ of a bundle $\cE$ over $X_e$ is
$ p_{\cE}(t):=\chi(\cE(th))/\rk(\cE).$

We say that a vector bundle $\cE$ is {\sl $\mu$-stable} (resp. {\sl $\mu$-semistable}) with respect to $H$ if  $\mu( \cG) < \mu(\cE)$ (resp. $\mu(\cG) \le \mu(\cE)$) for each subsheaf $ \cG$ with $0<rk(\cG)<rk(\cE)$.

On the other hand, $\cE$ is said to be Gieseker semistable  with respect to $H$ if for all $ \cG$ as above one has
$$
p_{ \cG}(t) \leq  p_{\cE}(t),
$$
and Gieseker stable again if equality cannot hold in the above inequality.

In order to check $\mu$-semistability of sheaves we will use the following criterion.
\begin{proposition}\cite[Theorem 3]{JMPS}\label{hoppe}
Let $\cE$ be a rank two vector bundle over a polycyclic variety $X$ and let $L$ be a polarization on $X$. $\cE$ is $\mu$-(semi)stable if and only if
\[
H^0(X,\cE\otimes \cO_X(B))=0
\]
for all $B \in Pic(X)$ such that $\delta_L(B) \underset{(<)}{\leq} -\mu_L(\cE)$, where $\delta_L(B)=\deg_L(\cO_X(B))$.
\end{proposition}
\begin{remark}\label{rem}
In our case $$\mu_H(\cE)=((e-1)f)H^2/2=(e-1)f(\xi+f)^2/2=(e-1)\xi^2 f+2(e-1)\xi f^2/2=$$ $$=(e^2-e+2e-2)/2=(e^2+e-2)/2$$ and $$\delta_H(a\xi+bf)=a\xi^3+2a\xi^2 f+a\xi f^2+b\xi^2 f+2b\xi f^2=ae^2+(2a+b)e+a+2b$$ so we get that $\cE$ is $\mu$-stable if and only if
$$H^0(S,\cE(a\xi+bf))=0$$ when $ae^2+(2a+b)e+a+2b\leq -(e^2+e-2)/2$. In particular the condition $H^0(X,\cE)=0$ included in the definition of $H$-instanton bundle is implied by the $\mu$-semistability only when $e=0$. Moreover if $a=-2, b=e$ we get $-2e^2-4e+e^2-2+2e=-e^2-2e-2\leq -(e^2+e-2)/2$ and $$H^0(X,\cE(-2\xi+ef))=0.$$
\end{remark}

\begin{lemma}\label{lem1}

Let $\cE$ a $H$-instanton bundle on $X_e$.
\begin{enumerate}
\item[(i)] $H^0(\cE(a\xi+bf))=0$ if $a\leq -1$, $b\leq e$ and $H^0(\cE(bf))=0$ if $b\leq 0$.
\item[(ii)] $H^3(\cE(a\xi+bf))=0$ for any $a\geq -1$, $b\geq -(e+2)$ and $H^3(\cE(-2\xi+bf))=0$ for any $b\geq -2$
\item[(iii)] $H^0(\Omega\otimes\cE(a\xi+bf))=0$ if $a\leq -1$, $b\leq e+1$ and $H^0(\cE(bf))=0$ if $b\leq 1$.
\item[(iv)] $H^3(\Omega\otimes\cE(a\xi+bf))=0$ for any $a\geq -1$, $b\geq -e$ and $H^3(\cE(-2\xi+bf))=0$ for any $b\geq 0$

\end{enumerate}
\end{lemma}
\begin{proof}
\begin{enumerate}
\item[(i)] Since $H^0(\cE)=0$, by (\ref{OmD}), we get $H^0(\cE(bf))=0$ for any $b\leq 0$.\\
Moreover, for Remark \ref{rem}, $H^0(\cE(-2\xi+ef))=0$ so, by (\ref{Rel}) tensored by $\cE$, we get $H^0(\cE(-\xi+ef))=0$. By (\ref{OmD}) we get $H^0(\cE(-\xi+bf))=0$ for any $b\leq e$. Finally by (\ref{OmD}) and (\ref{Om}) we get $(i)$.

\item[(ii)] By Serre duality, since $H^0(\cE)=0$, we get $h^3(\cE(-2\xi-2f))=0$; hence $H^3(\cE(-2\xi+bf))=0$ for any $b\geq -2$. By Serre duality, since $H^0(\cE(-\xi+ef))=0$, we get $h^3(\cE(-\xi-(e+2)f))=0$; hence $H^3(\cE(-\xi+bf))=0$ for any $b\geq -(e+2)$ and, by (\ref{Rel}) and (\ref{Om}), we get $(ii)$.\\

\item[(iii)] By sequence (\ref{Om}) and $(i)$, we get $(iii)$.

\item[(iv)] By sequence (\ref{OmD}) and $(ii)$, we get $(iv)$.\\

\end{enumerate}
\end{proof}

\begin{lemma}\label{lem2}

Let $\cE$ a $H$-instanton bundle on $X_e$.
\begin{enumerate}
\item[(1)] $H^2(\cE(a\xi+bf))=0$ if $a,b\geq -1$.

\item[(2)] $H^2(\Omega\otimes\cE(a\xi+bf))=0$ if $a\geq -1$, $b\geq 1$.

\end{enumerate}
\end{lemma}
\begin{proof}
\begin{enumerate}
\item[(1)] Since $H^2(\cE(-\xi-f)=0$ and $H^3(\Omega\otimes\cE(-\xi))=0$ for Lemma \ref{lem1} $(i)$, by (\ref{Om}) tensored by $\cE(-\xi)$, we get $H^2(\cE(-\xi))=0$  and by a recursive argument $H^2(\cE(-\xi+bf))=0$ for any $b\leq -1$.\\
Moreover, by (\ref{Rel}) tensored by $\cE(-f)$, we get $H^2(\cE(-f))=0$. Finally by (\ref{OmD}) by recursive argument we get $(1)$.

\item[(2)] By (\ref{Om}) and $(1)$, we get $(2)$.

\item[(ii)] By Serre duality, since $H^0(\cE)=0$, we get $h^3(\cE(-2\xi-2f))=0$; hence $H^3(\cE(-2\xi+bf))=0$ for any $b\geq -2$. By Serre duality, since $H^0(\cE(-\xi+ef))=0$, we get $h^3(\cE(-\xi-(e+2)f))=0$; hence $H^3(\cE(-\xi+bf))=0$ for any $b\geq -(e+2)$ and, by (\ref{Rel}) and (\ref{Om}), we get $(ii)$.\\

\item[(iv)] By (\ref{OmD}) and $(ii)$, we get $(iv)$.\\

\end{enumerate}
\end{proof}

 \section{Monads for $H$-instanton bundles}

Let us consider the full exceptional collection (see \cite{Orlov})
\begin{gather}\label{col1}
(\cE_5,\cE_4,\cE_3,\cE_2,\cE_2,\cE_0):=(\cO_X(-\xi-f)[-2],\cO_X(-\xi)[-2],\\
\cO_X(-\xi+f)[-2],
\cO_X(-(e+1)f),\cO_X(-ef),\cO_X(-(e-1)f),\notag
\end{gather}
\begin{gather}\label{col2}
(\cF_5,\cF_4,\cF_3,\cF_2,\cF_1,\cF_0):=(\cO_X(-\xi+(e-2)f),\pi^*\Omega^1_{\pp^2}(-\xi+ef),\\
\cO_X(-\xi+(e-1)f),
\cO_X((e-2)f),\pi^*\Omega^1_{\pp^2}(ef),\cO_X((e-1)f)).\notag
\end{gather}

We will use also the full exceptional collection (see \cite{Orlov})
\begin{gather}\label{col3}
(\cE_5,\cE_4,\cE_3,\cE_2,\cE_2,\cE_0):=(\cO_X(-\xi-f)[-2],\pi^*\Omega^1_{\pp^2}(-\xi+f)[-2],\\
\cO_X(-\xi)[-2],
\cO_X(-(e+1)f),\pi^*\Omega^1_{\pp^2}(-(e-1)f),\cO_X(-ef)).\notag
\end{gather}
\begin{gather}\label{col4}
(\cF_5,\cF_4,\cF_3,\cF_2,\cF_1,\cF_0): =(\cO_X(-\xi+(e-2)f),\cO_X(-\xi+(e-1)f),\\
\cO_X(-\xi+ef),
\cO_X((e-2)f),\cO_X((e-1)f),\cO_X(ef)),\notag
\end{gather}
and
\begin{gather}\label{col5}
(\cE_5,\cE_4,\cE_3,\cE_2,\cE_2,\cE_0):=(\cO_X(-\xi-f)[-2],\pi^*\Omega^1_{\pp^2}(-\xi+f)[-2],\\
\cO_X(-\xi)[-2],
\cO_X(-(e+2)f),\pi^*\Omega^1_{\pp^2}(-ef),\cO_X(-(e+1)f)).\notag
\end{gather}
\begin{gather}\label{col6}
(\cF_5,\cF_4,\cF_3,\cF_2,\cF_1,\cF_0): =(\cO_X(-\xi+(e-2)f),\cO_X(-\xi+(e-1)f),\\
\cO_X(-\xi+ef),
\cO_X((e-1)f),\cO_X(ef),\cO_X((e+1)f)),\notag
\end{gather}
For each sheaf $\cG$ and integer $j$ we denote by $\cG[j]$ the complex $G$ such that
$$
G_i=\left\lbrace\begin{array}{ll} 
0\quad&\text{if $i\ne j$,}\\
\cG\quad&\text{if $i=j$,}
\end{array}\right.
$$
with the trivial differentials: we will omit $[0]$ in the notation.

We shows that (\ref{col2}) is also strong. In fact  for $i>0$ 
$$Ext^i(\cO_X(-\xi+(e-2)f),\pi^*\Omega^1_{\pp^2}(-\xi+ef))=H^i(\pi^*\Omega^1_{\pp^2}(+2f))=0$$
 by (\ref{OmD}) twisted by $\cO_X(2f)$;
$$Ext^i(\cO_X(-\xi+(e-2)f),\cO_X(-\xi+(e-1)f))=H^i(\cO_X(f))=0;$$
$$Ext^i(\cO_X(-\xi+(e-2)f),\cO_X((e-2)f))=H^i(\cO_X(\xi))=0;$$
$$Ext^i(\cO_X(-\xi+(e-2)f),\pi^*\Omega^1_{\pp^2}(ef))=H^i(\pi^*\Omega^1_{\pp^2}(\xi+2f))=0$$
 by (\ref{OmD}) twisted by $\cO_X(\xi+2f)$;
$$Ext^i(\cO_X(-\xi+(e-2)f),\cO_X((e-1)f))=H^i(\cO_X(\xi+f))=0.$$
$$Ext^i(\pi^*\Omega^1_{\pp^2}(-\xi+ef),\cO_X(-\xi+(e-1)f))=H^i((\pi^*\Omega^1_{\pp^2})^\vee(-f))=0$$
 by the dual of (\ref{OmD}) twisted by $\cO_X(-f)$;
$$Ext^i(\pi^*\Omega^1_{\pp^2}(-\xi+ef),\cO_X((e-2)f))=H^i((\pi^*\Omega^1_{\pp^2})^\vee(\xi-2f))=0$$
by the dual of (\ref{OmD}) twisted by $\cO_X(\xi-2f)$;
$$Ext^i(\pi^*\Omega^1_{\pp^2}(-\xi+ef),\pi^*\Omega^1_{\pp^2}(ef))=H^i((\pi^*\Omega^1_{\pp^2})^\vee\otimes\pi^*\Omega^1_{\pp^2}(\xi))=0$$
 by (\ref{OmD}) twisted by $\pi^*\Omega^1_{\pp^2}(\xi)$;
$$Ext^i(\pi^*\Omega^1_{\pp^2}(-\xi+ef),\cO_X((e-1)f))=H^i((\pi^*\Omega^1_{\pp^2})^\vee(\xi-f))=0$$
by the dual of (\ref{OmD}) twisted by $\cO_X(\xi-f)$.

$$Ext^i(\cO_X(-\xi+(e-1)f),\cO_X((e-2)f))=H^i(\cO_X(\xi-f))=0;$$
$$Ext^i(\cO_X(-\xi+(e-1)f),\pi^*\Omega^1_{\pp^2}(ef))=H^i(\pi^*\Omega^1_{\pp^2}(\xi+f))=0$$
 by (\ref{OmD}) twisted by $\cO_X(\xi+f)$;
$$Ext^i(\cO_X(-\xi+(e-1)f),\cO_X((e-1)f))=H^i(\cO_X(\xi))=0.$$
$$Ext^i(\cO_X((e-2)f),\pi^*\Omega^1_{\pp^2}(ef))=H^i(\pi^*\Omega^1_{\pp^2}(2f))=0$$
 by (\ref{OmD}) twisted by $\cO_X(2f)$;
$$Ext^i(\cO_X((e-2)f),\cO_X((e-1)f))=H^i(\cO_X(f))=0.$$
$$Ext^i(\pi^*\Omega^1_{\pp^2}(ef),\cO_X((e-1)f))=H^i((\pi^*\Omega^1_{\pp^2})^\vee(-f))=0$$
by the dual of (\ref{OmD}) twisted by $\cO_X(-f)$.

With a very similar proof it is possible to show that sequences (\ref{col4}) and  (\ref{col6}) are also strong.\\
Now we able to construct monads associated to $H$-instanton bundles:

\begin{theorem}
\label{tSimplify}
Let $\cE$ be an instanton bundle with  $\gamma=h^2(\cE(-(e+1)f))=0$ and $c_2(\cE)=\alpha\xi^2+\beta f^2$ on $X_e$. 

Then $\cE$ is the cohomology of a monad  of the form
\scriptsize
$$
 0\rightarrow
\begin{matrix}
\pi^*\Omega^1_{\pp^2}(-\xi+ef)^{\alpha} \\
\oplus \\
\cO_X((e-2)f)^{\beta +\frac{-e^2+e}{2}}
\end{matrix}
\rightarrow
\begin{matrix}
\pi^*\Omega^1_{\pp^2}(ef)^{\alpha + \beta + \frac{-e^2+3e-2}{2}} \\
\oplus \\
\cO_X(-\xi+(e-1)f)^{2\alpha}
\end{matrix}
\rightarrow
\cO_F((e-1)f)^{2 \alpha + \beta + \frac{-e^2 +5e - 8}{2}}
\rightarrow
0.
 $$

\end{theorem}

\begin{proof}

We construct a Beilinson complex, quasi-isomorphic to $\cE$, by calculating $H^{i+k_j}(\cE\otimes \cE_j)\otimes \cF_j$ with  $i,j \in \{0, \ldots, 5\}$, using (\ref{col1}) and (\ref{col2}), to get the following table:

 \begin{center}
 \scriptsize
 \begin{tabular}{|c|c|c|c|c|c|c|c|c|c|c|}
\hline

 $\cO(-\xi+(e-2)f)$ & $\Omega(-\xi+ef)$ & $\cO(-\xi+(e-1)f)$ & $\cO((e-2)f)$ & $\Omega(ef)$ & $\cO((e-1)f)$ \\
 \hline
 \hline
$H^3$	&	$H^3$	&	$H^3$		&	$*$		&	$*$		&	$*$	\\
\hline
$H^2$	&	$H^2$	&	$H^2$	&	$*$	&	$*$		&	$*$	\\
\hline
$H^1$	& 	$H^1$	&	$H^1$	&	$H^3$	&	$H^3$	&	$H^3$	\\
\hline
$H^0$	& 	$H^0$	&	$H^0$	&	$H^2$	&	$H^2$	&	$H^2$	\\
\hline
$*$		&	$*$ 	 	&	$*$	&	$H^1$	& 	$H^1$	& 	$H^1$	\\
\hline
$*$		&	$*$		&	$*$		&	$H^0$		&	$H^0$	& 	$H^0$ \\
\hline
\hline
$\cO(-\xi-f)$ & $\cO(-\xi)$ & $\cO(-\xi+f)$ & $\cO(-(e+1)f)$ & $\cO(-ef)$ & $\cO(-(e-1)f)$ \\
\hline
\end{tabular}
\end{center}

We called $\pi^*\Omega^1_{\pp^2}$ simply $\Omega$ and $\cO_X$ simply $\cO$.\\
By Lemma \ref{lem1}, all the $H^0, H^3$ in the table are zero  (for the case $H^0(\cE(-\xi+f))=0$ when $e=0$ not included in the Lemma see \cite{AM2} Remark 2.5) .\\ Moreover, the first column is zero by the instantonic condition and 
$H^2(\cE(-\xi))=H^2(\cE(-\xi+f))=0$ by Lemma \ref{lem2}.\\ 
Finally by  (\ref{Om}) tensored by $\cE(-ef)$, since $\gamma=h^2(\cE(-(e+1)f))=0$ and $H^3(\cE\otimes\Omega(-ef))=0$ we get $H^2(\cE(-ef))=0$ and similarly $H^2(\cE(-(e-1)f))=0$. 
So the table becomes

 \begin{center}
 \scriptsize
 \begin{tabular}{|c|c|c|c|c|c|c|c|c|c|c|}
\hline

 $\cO(-\xi+(e-2)f)$ & $\Omega(-\xi+ef)$ & $\cO(-\xi+(e-1)f)$ & $\cO((e-2)f)$ & $\Omega(ef)$ & $\cO((e-1)f)$ \\
 \hline
 \hline
$0$	&	$0$	&	$0$		&	$*$		&	$*$		&	$*$	\\
\hline
$0$	&	$0$	&	$0$	&	$*$	&	$*$		&	$*$	\\
\hline
$0$	& 	$H^1$	&	$H^1$	&	$0$	&	$0$	&	$0$	\\
\hline
$0$	& 	$0$	&	$0$	&	$0$	&	$0$	&	$0$	\\
\hline
$*$		&	$*$ 	 	&	$*$	&	$H^1$	& 	$H^1$	& 	$H^1$	\\
\hline
$*$		&	$*$		&	$*$		&	$0$		&	$0$	& 	$0$ \\
\hline
\hline
$\cO(-\xi-f)$ & $\cO(-\xi)$ & $\cO(-\xi+f)$ & $\cO(-(e+1)f)$ & $\cO(-ef)$ & $\cO(-(e-1)f)$ \\
\hline
\end{tabular}
\end{center}
By the equality (\ref{RRB}), we have $h^1(\cE(-\xi)) = \alpha$,\\ $h^1(\cE(-\xi +f)) = 2\alpha $,\\ $h^1(\cE(-(e+1)f)) = \beta + \frac{-e^2+e}{2}$,\\ $h^1(\cE(-ef)) = \alpha + \beta + \frac{-e^2+3e-2}{2}$\\ and $h^1(\cE(-(e-1)f)) = 2 \alpha + \beta + \frac{-e^2 +5e - 8}{2}$,\\
and we obtain the claimed monad.
\end{proof}

\begin{remark}
    For $e=1$ we have the monad of \cite{CCGM} Theorem 1.3.
\end{remark}

\begin{remark}\label{pbk} We obtain the following bounds for $\alpha $ and $\beta $:
\begin{itemize}
    \item $\alpha \geq 0.$
    \item If $\alpha =0$, $\beta \geq -1 + \frac{e^2+e}{2}$, $\beta \geq - \frac{-e^2+3e-2}{2}$, $\beta \geq -\frac{-e^2 +5e - 8}{2}$.\\
    Moreover we get $h^i(\cE(-\xi +f)) =h^i(\cE(-\xi)) =0$ for any $i$. So by (\ref{Omt}) tensored by $\cE(-\xi+f)$ we get $h^i(\cE(-\xi -2f)) =0$ for any $i$. Now, by (\ref{Omt}) tensored by $\cE(-\xi+2f)$ we get $h^i(\cE(-\xi +2f)) =0$ for any $i$ and by (\ref{Omt}) tensored by $\cE(-\xi)$ we get $h^i(\cE(-\xi -3f)) =0$ for any $i$. Then, by a recursive argument, we may conclude that $h^i(\cE(-\xi -tf)) =0$ for any $i$ and for any integer $t$.\\
    By (\ref{Rel}) tensored by $\cE(bf)$, since by Lemma \ref{lem1} $H^3(\cE(-2\xi+bf))=0$ for any $b\geq -2$, we get $h^2(\cE(bf))=0$ for any $b\geq -(e+2)$. In particular $\cE$ is earnest.
\end{itemize}

\end{remark}
We can now characterize earnest $H$-instanton bundles:
\begin{theorem}\label{thm:earnest}
		Let $\cE$ be a $H$-instanton on $X_e$. Then, $\cE$ is earnest if and only if $\h2 (\cE (-(e+1)f))=0$.
	\end{theorem}
	
	\begin{proof}
		If $\cE$ is earnest, then $\h2 (\cE (-(e+1)f))$ by Remark \ref{ear}.
		
		Conversely, let us assume that $\h2 (\cE (-(e+1)f))=\h1 (\cE(-h-E)) =0$. Then, it remains to prove that $\h1 (\cE (-(1+a)\xi - (1+b)f)) = 0$ for $a,b \geq 0$ by Proposition \ref{prop: smooth int. divisor}. Since $\cE$ be a $H$-instanton on $X_e$ and $\h2 (\cE (-(e+1)f))=0$, $\cE$ is the cohomology of the monad in Theorem \ref{tSimplify} . If we tensor the monad with the line bundle $\cL = \cO_{X_e} \big(-(1+a)\xi - (1+b)f \big)$, we have
		
		\begin{eqnarray}\label{ineq:bound}
			\h1 \big(F,\cE\otimes\mathcal L\big) \le \sum_{j=-1}^1 \h{1-j}\big(F,\cC^{j}\otimes\mathcal L\big)
		\end{eqnarray}
		
		where $\cC^{-1} = \pi^*\Omega^1_{\pp^2}(-\xi+ef)^{\alpha} \oplus
		\cO_X((e-2)f)^{\beta +\frac{-e^2+e}{2}}$,\\
		 $\cC^{0} = 	\pi^*\Omega^1_{\pp^2}(ef)^{\alpha + \beta + \frac{-e^2+3e-2}{2}} \oplus \cO_X(-\xi+(e-1)f)^{2\alpha}$ and\\
		  $\cC^{1} = \cO_F((e-1)f)^{2 \alpha + \beta + \frac{-e^2 +5e - 8}{2}}$. Then, the right hand side of the inequality (\ref{ineq:bound}) vanishes thanks to Lemma \ref{lem0} and the short exact sequence (\ref{Om}).
	\end{proof}

\begin{remark}
If $\gamma\not=0$, then $\cE$ is not earnest and it is the cohomology of a monad  of the form
    \begin{equation*}
\label{Monad}
0\longrightarrow \cA\longrightarrow \cB\longrightarrow\cC\longrightarrow0
\end{equation*}
where

\begin{gather*}
\cA:=\pi^*\Omega^1_{\pp^2}(-\xi+ef)^{\alpha}\oplus\cO_X((e-2)f)^{\beta + \frac{-e^2+e}{2}-\gamma},\\
\cB:=\pi^*\Omega^1_{\pp^2}(ef)^{\alpha + \beta + \frac{-e^2+3e-2}{2}-\delta}\oplus\cO_X(-\xi+(e-1)f)^{2\alpha}\oplus \cO_X((e-2)f)^{\gamma},\\
\end{gather*}
and $\cC$ arises from the following exact sequence:
\begin{gather*}
0\to\cC\to\cO_F((e-1)f)^{2 \alpha + \beta + \frac{-e^2 +5e - 8}{2}-\eta}\oplus\pi^*\Omega^1_{\pp^2}(ef)^{\delta}\to\cO_F((e-1)f)^{\eta}\to 0,
\end{gather*}
with $\eta:=h^2(\cE(-ef))$ and $\delta:=h^2(\cE(-(e-1)f))$. 
\end{remark}

We can construct a second monad:
\begin{theorem}\label{mon2}
Let $\cE$ be an instanton bundle with  $\gamma=h^2(\cE(-(e+1)f))=0$ and $c_2(\cE)=\alpha\xi^2+\beta f^2$ on $X_e$. 

Then $\cE$ is the cohomology of a monad  of the form
\small
$$
 0\rightarrow
\begin{matrix}
\cO_X(-\xi+(e-1)f)^{\alpha} \\
\oplus \\
\cO_X((e-2)f)^{\beta + \frac{e^2 +e}{2}}
\end{matrix}
\rightarrow
\begin{matrix}
\cO_X((e-1)f)^{2\beta+\alpha-e^2+2e+1} \\
\oplus \\
\cO_X(-\xi+ef)^{\alpha}
\end{matrix}
\rightarrow
\cO_X(ef)^{\alpha+\beta + \frac{-e^2 +3e-2}{2}}
\rightarrow
0.
 $$


\end{theorem}

\begin{proof}

We construct a Beilinson complex, quasi-isomorphic to $\cE$, by calculating $H^{i+k_j}(\cE\otimes \cE_j)\otimes \cF_j$ with  $i,j \in \{0, \ldots, 5\}$, using (\ref{col3}) and (\ref{col4}).
We called $\pi^*\Omega^1_{\pp^2}$ simply $\Omega$ and $\cO_X$ simply $\cO$.\\
By Lemma \ref{lem1} , (\ref{OmD}), all the $H^0, H^3$ in the table are zero.\\ Moreover the first column is zero by the instantonic condition.

By  Lemma \ref{lem2} get $H^2(\cE(-\xi))=H^2(\cE\otimes\Omega(-\xi+f))=0$.\\
Finally by  (\ref{Om}) tensored by $\cE(-ef)$, since $\gamma=h^2(\cE(-(e+1)f))=0$ and $H^3(\cE\otimes\Omega(-ef))=0$ we get $H^2(\cE(-ef))=0$. Similarly we obtain $H^2(\cE\otimes\Omega(-(e-1)f))=0$.
So the table becomes

  \begin{center}
 \scriptsize
 \begin{tabular}{|c|c|c|c|c|c|c|c|c|c|c|}
\hline

$\cO_X(-\xi+(e-2)f)$ & $\cO_X(-\xi+(e-1)f)$ & $\cO_X(-\xi+ef)$
& $\cO_X((e-2)f)$ & $\cO_X((e-1)f)$ & $\cO_X(ef)$ \\
 \hline
 \hline
$0$	&	$0$	&	$0$		&	$*$		&	$*$		&	$*$	\\
\hline
$0$	&	$0$	&	$0$	&	$*$	&	$*$		&	$*$	\\
\hline
$0$	& 	$H^1$	&	$H^1$	&	$0$	&	$0$	&	$0$	\\
\hline
$0$	& 	$0$	&	$0$	&	$0$	&	$0$	&	$0$	\\
\hline
$*$		&	$*$ 	 	&	$*$	&	$H^1$	& 	$H^1$	& 	$H^1$	\\
\hline
$*$		&	$*$		&	$*$		&	$0$		&	$0$	& 	$0$ \\
\hline
\hline
$\cO_X(-\xi-f)$ & $\Omega (-\xi+f)$ & $\cO_X(-\xi)$ &
$\cO_X(-(e+1)f)$ & $\Omega (-(e-1)f)$ & $\cO_X(-ef)$ \\
\hline
\end{tabular}
\end{center}
By the equality (\ref{RRB}), we have $h^1(\cE(-\xi)) = \alpha$,\\  $h^1(\cE(-(e+1)f)) = \beta + \frac{-e^2+e}{2}$\\ and $h^1(\cE(-(e-1)f)) = \alpha+\beta + \frac{-e^2 +3e-2}{2}$.\\
By  (\ref{Om}), we get $h^1(\cE\otimes\Omega(-\xi+f))=3h^1(\cE(-\xi))-h^1(\cE(-\xi+f))=3\alpha-2\alpha=\alpha$.\\ By  (\ref{OmD}), we get $h^1(\cE\otimes\Omega(-(e-1)f))=3h^1(\cE(-ef))-h^1(\cE(-(e-1)f))=3 (\alpha+\beta + \frac{-e^2 +3e-2}{2})-(2 \alpha + \beta + \frac{-e^2 +5e - 8}{2})=2\beta+\alpha-e^2+2e+1$,\\
and we obtain the claimed monad.
\end{proof}

We now give a slightly modified and simplified version in the case $\alpha=0$:
\begin{proposition}\label{mon3}
Let $\cE$ be an $H$-instanton bundle with  $\alpha=0$  on $X_e$. 

Then $\cE$ is the cohomology of a monad  of the form

 $$
 0\rightarrow\cO_X((e-1)f)^{\beta - \frac{e^2 +e}{2}-1}
\rightarrow\cO_X(ef)^{2\beta-e^2+1}\rightarrow\cO_X((e+1)f)^{\beta + \frac{-e^2 +e}{2}}\rightarrow 0.
 $$

\end{proposition}

\begin{proof}

We construct a Beilinson complex, quasi-isomorphic to $\cE$, by calculating $H^{i+k_j}(\cE\otimes \cE_j)\otimes \cF_j$ with  $i,j \in \{0, \ldots, 5\}$using (\ref{col5}) and (\ref{col6}).

By Lemma \ref{lem1} and Lemma \ref{lem2} , as in the above Theorem, all the $H^0, H^3$ in the table are zero and $H^2(\cE(-\xi))=H^2(\cE\otimes\Omega(-\xi+f))=0$.\\
When $\alpha =0$ we get $h^2(\cE(-(e+2)f))=0$, by Remark \ref{pbk} and $H^3(\cE\otimes\Omega(-ef))=0$, we get $H^2(\cE(-ef))=0$. Similarly $H^2(\cE\otimes\Omega(-ef))=0$. 
So the table becomes:

  \begin{center}
 \scriptsize
 \begin{tabular}{|c|c|c|c|c|c|c|c|c|c|c|}
\hline

$\cO_X(-\xi+(e-2)f)$ & $\cO_X(-\xi+(e-1)f)$ & $\cO_X(-\xi+ef)$
& $\cO_X((e-1)f)$ & $\cO_X(ef)$ & $\cO_X((e+1)f)$ \\
 \hline
 \hline
$0$	&	$0$	&	$0$		&	$*$		&	$*$		&	$*$	\\
\hline
$0$	&	$0$	&	$0$	&	$*$	&	$*$		&	$*$	\\
\hline
$0$	& 	$H^1$	&	$H^1$	&	$0$	&	$0$	&	$0$	\\
\hline
$0$	& 	$0$	&	$0$	&	$0$	&	$0$	&	$0$	\\
\hline
$*$		&	$*$ 	 	&	$*$	&	$H^1$	& 	$H^1$	& 	$H^1$	\\
\hline
$*$		&	$*$		&	$*$		&	$0$		&	$0$	& 	$0$ \\
\hline
\hline
$\cO_X(-\xi-f)$ & $\Omega (-\xi+f)$ & $\cO_X(-\xi)$ &
$\cO_X(-(e+2)f)$ & $\Omega (-ef)$ & $\cO_X(-(e+1)f)$ \\
\hline
\end{tabular}
\end{center}
By the equality (\ref{RRB}), we have $h^1(\cE(-\xi)) = \alpha=0$,\\  $h^1(\cE(-(e+1)f)) = \beta + \frac{-e^2+e}{2}$,\\ $h^1(\cE(-ef)) =  \beta + \frac{-e^2+3e-2}{2}$\\ and $h^1(\cE(-(e+2)f)) = \beta - \frac{e^2 +e}{2}-1$.\\
By  (\ref{Om}), we get $h^1(\cE\otimes\Omega(-\xi+f))=\alpha=0$.\\ By  (\ref{OmD}), we get $h^1(\cE\otimes\Omega(-ef))=3h^1(\cE(-(e+1)f))-h^1(\cE(-ef))=3 (\beta + \frac{-e^2+e}{2})-( \beta + \frac{-e^2+3e-2}{2})=2\beta-e^2+1$,
and we obtain the claimed monad.
\end{proof}

\begin{remark} 

Thanks the above proposition we may conclude that
   $\cE=\pi^*\cF$ where $\cF$ is the homology of the following monad over $\pp^2$:
  \begin{equation}\label{brun}
 0\rightarrow\cO_{\pp^2}(e-1)^{\beta - \frac{e^2 +e}{2}-1 }
\rightarrow\cO_{\pp^2}(e)^{2\beta-e^2+1}\rightarrow\cO_{\pp^2}(e+1)^{\beta + \frac{-e^2 +e}{2}}\rightarrow 0.
   \end{equation}
  In particular when $\beta$ is minimal, namely $ \beta =  \frac{e^2 +e}{2}+1$, we get

$2\beta-e^2+1=2(\frac{e^2 +e}{2}+1)-e^2+1=e+3$ and
$\beta + \frac{-e^2 +e}{2}=\frac{e^2 +e}{2}+1+\frac{-e^2+e}{2}=e+1$ so we obtain

 $$
 0\rightarrow
\cF
\rightarrow\cO_{\pp^2}(e)^{e+3}\rightarrow\cO_{\pp^2}(e+1)^{e+1}\rightarrow 0.
 $$
We have that $\cE=\pi^*\cF(H)$ is Ulrich on $X_e$. Since $\cF$ exists by \cite{CG}, also $\cE$ for any $e\geq 0$. \end{remark}

\begin{remark}
We show that the bundles over $\pp^2$ arising from monad (\ref{brun}) actually exist. Let us consider the case when $e$ is even or odd:
\begin{itemize}
    \item If $e=2t+1$, $c_1(\cF(-t))=0$ and $c_2(\cF(-t))=c_1(\cF)+c_1(\cF)(-t)+(-t)^2=\beta -t(e-1)+t^2=\beta-\frac{(e-1)^2}{4}$. So $\cG=(\cF(-t))$ belongs to the moduli space $\cM(0,\beta -t(e-1)+t^2)$ studied in \cite{Bar} with the extra condition $H^0(\cG(t))=0$. In \cite{Br} it is proven that the generic stable rank two bundle on $\pp^2$ has at most one non-zero cohomology group, so, when $\chi(\cG(t))\leq 0$ we must have $H^0(\cG(t))=0$. Let us compute $$\chi(\cG(t))=(t+1)(t+2)-c_2(\cG)=(t+1)(t+2)-(\beta -t(e-1)+t^2)\leq 0,$$ hence $$\beta\geq (\frac{e-1}{2}+1)(\frac{e-1}{2}+2) +\frac{e-1}{2}(e-1)-(\frac{e-1}{2})^2=$$ $$=3\frac{e-1}{2}+2+\frac{e-1}{2}(e-1)=e-1+2+\frac{e^2-e}{2}=1+\frac{e^2+e}{2}.$$ So when $\beta\geq 1+\frac{e^2+e}{2}$ there exist $H$-instanton bundles $\cE$ on $X_e$ with $c_2=\beta f^2$.

     \item If $e=2t$, $c_1(\cF(-t))=-1$ and $c_2(\cF(-t))=c_1(\cF)+c_1(\cF)(-t)+(-t)^2=\beta -t(e-1)+t^2=\beta+\frac{2e-e^2}{4}$. So $\cG=(\cF(-t))$ belongs to the moduli space $M(-1,\beta -t(e-1)+t^2)$ studied in \cite{Hu} with the extra condition $H^0(\cG(t))=0$. In \cite{Br} it is proven that the generic stable rank two bundle on $\pp^2$ has at most one non-zero cohomology group, so, when $\chi(\cG(t))\leq 0$ we must have $H^0(\cG(t))=0$. Let us compute $$\chi(\cG(t))=(t+1)^2-c_2(\cG)=(t+1)^2-(\beta -t(e-1)+t^2)\leq 0,$$ hence $$\beta\geq (\frac{e}{2}+1)^2 +\frac{e}{2}(e-1)-(\frac{e}{2})^2=$$ $$=e+1+\frac{e}{2}(e-1)=1+\frac{e^2+e}{2}.$$ So when $\beta\geq 1+\frac{e^2+e}{2}$ there exist $H$-instanton bundles $\cE$ on $X_e$ with $c_2=\beta f^2$.

\end{itemize}

 Recall that the moduli space of $\mu$--stable vector bundles $\cF$ on $\pp^2$ of rank $2$ with $c_1(\cF)=0$ or $c_1(\cF)=-1$ and $c_2(\cF)=k>1$ is integral, rational and smooth of dimension $4k-3$ and $4k-4$ (see \cite{Bar} and \cite{Hu} respectively). It then follows  that the locus
$$
\cI_{X_e}^{earnest}(\beta f^2)=\cI_{X_e}(\beta f^2)\subseteq\cM_{X_e}(2;(e-1)f,\beta f^2)
$$
of earnest instanton bundles is integral, rational and smooth as well of dimension $4\beta+2e-e^2-4$. For the case $e=1$ see \cite{CCGM}.
\end{remark}

\section{Existence of $H$-instanton bundles}
Let \( M \) be a curve in the class \(\xi f\) within \( A^2(X_e) \). Since \( M \) is a complete intersection, its structure sheaf has the following resolution:
\begin{equation}\label{ses: resolution of M}
	0 \rightarrow \mathcal{O}_{X_e} (-\xi - f) \rightarrow \mathcal{O}_{X_e}(-\xi) \oplus \mathcal{O}_{X_e}(-f) \rightarrow \mathcal{O}_{X_e} \rightarrow \mathcal{O}_M \rightarrow 0.
\end{equation}
In particular, we have \(\mathcal{N}_{M|X_e} \simeq \mathcal{O}_M (1) \oplus  \mathcal{O}_M (e)\), so the determinant is 
\begin{equation}\label{rem:curve M}
    \det(\mathcal{N}_{M|X_e}) = \mathcal{O}_M (e+1).
\end{equation}

Moreover, the intersection number satisfies \(\xi f (\xi + f) = e + 1\), and we compute
\[
\chi(\mathcal{O}_M) = \chi(\mathcal{O}_{X_e}) - \chi(\mathcal{O}_{X_e} (-f)) - \chi(\mathcal{O}_{X_e} (-\xi)) + \chi(\mathcal{O}_{X_e} (-\xi - f)) = 1,
\]
since \(\chi(\mathcal{O}_{X_e} (-f)) = \chi(\mathcal{O}_{X_e} (-\xi)) = \chi(\mathcal{O}_{X_e} (-\xi - f)) = 0\) by Lemma \ref{lem0}. Therefore, \( M \) is a rational curve of degree \( e + 1 \) in \( X_e \).

From now on, we denote by \(\Lambda_M\) the component of the Hilbert scheme \({\rm Hilb}^{(e+1)t+1}(X_e)\) containing these complete intersection curves. Since
\[
h^0(\mathcal{N}_{M|X_e}) = e + 3 \quad \text{and} \quad h^1(\mathcal{N}_{M|X_e}) = 0,
\]
we conclude that \(\Lambda_M\) is smooth and has dimension \( e + 3 \). Note that generic elements of \(\Lambda_M\) do not intersect each other. If they did, consider two generic elements \(\xi_0 f_0\) and \(\xi_1 f_1\) of \(\Lambda_M\) intersecting at a point \( p \). Then the divisor \( f_1 \) intersects \(\xi_0 f_0\) at \( p \), and thus \(\xi_1\) passes through \( p \). This would imply that any element \(\xi_i f_1\) in \(\Lambda_M\) intersects \(\xi_0 f_0\) at \( p \), meaning \(\xi_i\) passes through \( p \) for any \( i \). However, this leads to a contradiction, because the line bundle \(\mathcal{O}_{X_e}(\xi)\) is globally generated.

	\begin{theorem} \label{thm:existence_serre}
		There exists a stable instanton bundle $\cE$ on $X_e$ with $c_2 = \alpha \xi f$ for $\alpha > e$. 
	\end{theorem}
	
	\begin{proof}
		Let  $M_1,\dots, M_{\alpha+1}\in \Lambda_M$ be general curves where $\alpha > e$. Notice that we can assume such curves pairwise disjoint, due to the definition of $\Lambda_M$. We define 
		\begin{equation}
			\label{XInstanton}
			Z:=\bigcup_{i=1}^{\alpha+1}M_i  \subset X_e.
		\end{equation}
		Consider the line bundle $\mathcal{L} \simeq \cO_{X_e}(2\xi +(1-e)f)$. Thanks to  equality \eqref{rem:curve M}, we have
		\begin{gather*}
			\det(\mathcal{N}_{Z\vert X_e})\otimes\cO_{M_i}\cong\cO_{M_i}(e+1) \cong \mathcal{L} \otimes\cO_{M_i}.
		\end{gather*}
		This isomorphism at each disjoint component ensures that $\det(\mathcal{N}_{Z\vert X_e}) \cong \mathcal{L} \otimes\cO_Z$.
		Hartshorne-Serre correspondence, \cite[Theorem 1.1]{Ar}, yields the existence of a vector bundle $\cF$ on $X_e$ with a section $s$ vanishing exactly along $Z$ and with $c_1(\cF)=2\xi +(1-e)f$, $c_2(\cF)= (\alpha+1) \xi f$, because the curves $M_i$ have been chosen pairwise disjoint and $h^2\big(X_e, \cO_{X_e}(-2 \xi + f)\big)=0$ by Lemma \ref{lem0}. Thus $\cE:=\cF(-\xi +(e-1)f)$ is a vector bundle fitting into an exact sequence of the form
		\begin{equation}\label{ses:constructed_bundle_by_Serre}
			0\longrightarrow\cO_{X_e}(-\xi +(e-1)f)\longrightarrow \cE \longrightarrow \mathcal{I}_{Z \vert X_c}(\xi) \longrightarrow 0.
		\end{equation}
		Notice that 
		\begin{eqnarray*}
			c_1 (\cE) &=& c_1(\cF) + 2(-\xi+(e-1)f) = (e-1)f\\
			c_2 (\cE) &=& c_2(\cF) + c_1(\cF) \cdot (-\xi+(e-1)f) + \big(-\xi+(e-1)f \big)^2 f^2 = \alpha \xi f.
		\end{eqnarray*}
		
		First, let us prove the vanishing of $\H0 (\cE)$. The cohomology of the exact sequence (\ref{ses:constructed_bundle_by_Serre}) yields $\h0( \cE) \leq \h0(\cO_{X_e} (-\xi + (e-1) f)) + \h0 (\mathcal{I}_{Z \vert X_e} (\xi))$. If $\alpha > e$ then $Z$ will not lie in an element of the class $|\xi|$; that is, $\h0 (\mathcal{I}_{Z \vert X_e} (\xi)) = 0$. Also, $\h0(\cO_{X_c} (-\xi + f))  =0$ by Lemma \ref{lem0}. So, $\h0(\cE) = 0$.
		
		Similarly, we have $\h1(\cE(-H)) \leq \h1(\cO_{X_c} (-2\xi  + (e-2f))) + \h1 (\cI_{Z \vert X_c} (-f))$. Again by Lemma \ref{lem0}, we have $\h1(\cO_{X_c}(-2\xi + (e-2)f))  = 0$. The cohomology of the following short exact sequence 
		\begin{equation}\label{ses:ideal_sheaf}
			0 \rightarrow \cI_{Z \vert X_e}  \rightarrow \cO_{X_c} \rightarrow \cO_{Z} \rightarrow 0
		\end{equation}
		after tensoring by $\cO_{X_e} (-f)$ gives us the equality $\h1 (\cI_{Z \vert X_e} (-f)) = \h0 (\cO_{Z} (-f))$ by the help of Lemma \ref{lem0}. Since $\mathrm{deg} (\xi f \cdot (-f)) = -1$, $\h0 (\cO_{Z} (-f)) = 0$. So, $\h1(\cE(-H)) = 0$.\\
		Now, let us prove that $\cE$ is $\mu$-stable. By Proposition \ref{hoppe}, it is enough to show that $\h0(\cE (a\xi + bf)) = 0$ for all $a,b\in\mathbb{Z}$ satisfying
		\begin{equation} \label{ineq:hoppe}
			(e+1)^2 a+(e+2)b \leq -\frac{(e+2)(e-1)}{2}.
		\end{equation}
		If we tensor sequences in display \eqref{ses:constructed_bundle_by_Serre} by $\cO_{X_c}(a\xi + bf)$ then the induced cohomology sequence gives us
		\begin{equation}\label{ineq:existence_theorem_1}
			\h0(\cE(a\xi + bf)) \leq \h0 (\cO_{X_c} ((a-1)\xi + (b+e-1)f)) + \h0 ( \cI_{Z}((a+1)\xi + bf)).
		\end{equation}
		Similarly, we tensor sequences in display \eqref{ses:ideal_sheaf} by $\cO_{X_c}((a+1)\xi + bf)$ then the induced cohomology sequence gives us 
		\begin{equation}\label{ineq:existence_theorem_2}
			\h0 ( \cI_{Z}((a+1)\xi + bf)) \leq \h0 (\cO_{X_c} ((a+1)\xi + bf)).
		\end{equation}
		
		Note that $\h0 (\cO_X((a-1)\xi + (b+e-1)f)) = 0$ for $a \leq 0$ and  $\h0 (\cO_{X_c} ((a+1)\xi + bf)) =0$ for $a \leq -2$ by Lemma \ref{lem0}. So, combining this with inequalities \eqref{ineq:existence_theorem_1} and \eqref{ineq:existence_theorem_2}, $\h0(\cE(a\xi + bf)) = 0$ for $a \leq -2.$
		
		If $a= -1$, then $\h0(\cE( -\xi + bf)) \leq  \h0 ( \cI_{Z}((bf))$ by the inequality \eqref{ineq:existence_theorem_1}, and for $b \leq \frac{e+1}{2}$ by the inequality \eqref{ineq:hoppe}. Since $\alpha \geq e$, $Z$ will not lie in an element of the class $|bf|$ for $b \leq \frac{e+1}{2}$; that is, $\h0 (\mathcal{I}_{Z \vert X_e} (bf)) = 0$. So, $\h0(\cE( -\xi + bf)) = 0$ for $b \leq \frac{e+1}{2}$.
		
		If $a \geq 0$, then we will use induction on $a$ to show that $\h0(\cE(a\xi + bf))$ for all $a,b\in\mathbb{Z}$ satisfying inequality \eqref{ineq:hoppe}.
		
		If $a =0$, then $b \leq 0$ by inequality \eqref{ineq:hoppe}. But we have already showed that $\h0 (\cE) =0$. Then, $\h0(\cE( bf))$ for any $b \leq 0$ as $f$ is an effective divisor.
		
		Assume that $\h0(\cE(k\xi + bf)) = 0$ for some $k \geq 1$ and $ b \leq \frac{1-e}{2} -\frac{k(e+1)^2}{e+2}$.
		
		Now, we will show that $\h0(\cE((k+1)\xi + bf)) = 0$ for $b\leq \frac{1-e}{2} -\frac{k(e+1)^2}{e+2} -e-\frac{1}{e+2}$. If we tensor sequence \eqref{Rel} by $\cE((k+1)\xi + bf)$
		then the induced cohomology sequence gives us
		\begin{equation*}
			\h0 ( \cE((k+1)\xi + bf)) \leq \h0 (\cE(k\xi + bf)) + \h0 ( \cE(k\xi + (b+e)f)) + \h1 (\cE((k-1)\xi + (b+e)f)).
		\end{equation*}
		Notice that $\h0 (\cE(k\xi + bf)) = \h0 ( \cE(k\xi + (b+e)f)) =0$ by the inductive step. So, we have 
		\begin{equation}
			\h0 ( \cE((k+1)\xi + bf)) \leq  \h1 (\cE((k-1)\xi + (b+e)f)).
		\end{equation}
		If we tensor sequence \eqref{ses:constructed_bundle_by_Serre} by $\cO_{X_e}((k-1)\xi + (b+e)f)$ then the induced cohomology sequence gives us
		\begin{equation*}
			\h1 (\cE((k-1)\xi + (b+e)f)) \leq \h1 (\cO_{X_e}((k-2)\xi + (b+2e-1)f)) + \h1 (\cI_{Z}(k\xi +(b+e)f)).
		\end{equation*}
		Notice that $\h1 (\cO_{X_e}((k-2)\xi + (b+2e-1)f)) = 0$ by Lemma \ref{lem0}. So, we have 
		\begin{equation*}
			\h1 (\cE((k-1)\xi + (b+e)f)) \leq \h1 (\cI_{Z}(k\xi +(b+e)f)).
		\end{equation*}
		If we tensor sequence \eqref{ses:ideal_sheaf} by $\cO_{X_e}(k\xi + (b+e)f)$ then the induced cohomology sequence gives us
		\begin{equation*}
			\h1 (\cI_{Z}(k\xi +(b+e)f)) \leq \h1 (\cO_{X_e}(k\xi + (b+e)f)) + \h0 (\cO_{Z}(k\xi + (b+e)f)).
		\end{equation*}
		Similarly, by $\h1 (\cO_{X_e}(k\xi + (b+e)f)) =0$ Lemma \ref{lem0}. Also, we have 
		\begin{eqnarray*}
			\xi f \cdot (k \xi +(b+e)f) = ke + b +e &\leq& \frac{1-e}{2} -\frac{k(e+1)^2}{e+2} -e-\frac{1}{e+2} + ke +e \\
			&\leq& \frac{1-e}{2} -\frac{k}{e+2} - \frac{1}{e+2} <0.
		\end{eqnarray*}
		So, $\h0 (\cO_{Z}(k\xi + (b+e)f)) = 0$. Then, $\h1 (\cE((k-1)\xi + (b+e)f)) = 0$. As a result, $\h0 ( \cE((k+1)\xi + bf) =0$, which completes the induction. Hence, $\cE$ is $\mu$-stable.
	\end{proof}
	
	\begin{theorem}\label{thm:ext_computation}
		Let $\cE$ be one of the $H$-instanton bundle on $(X_e , H)$ constructed in Theorem \ref{thm:existence_serre}. Then $\dim \operatorname{Ext}^0(\cE, \cE) =1$, $\dim \operatorname{Ext}^3 (\cE, \cE) =0$, $\dim \operatorname{Ext}^2 (\cE, \cE) = \begin{cases}
			0& \text{if } e \leq 3 \\
			\binom{e-2}{2}& \text{if } e\geq 4
		\end{cases}$ and $\dim \operatorname{Ext}^1 (\cE, \cE) - \dim \operatorname{Ext}^2 (\cE, \cE) = (6+2e) \alpha +4\beta -(e-1)^2 -3$.
	\end{theorem}
	
	\begin{proof}
		Since we have showed in Theorem \ref{thm:existence_serre} that $\cE$ is $\mu$-stable, $\cE$ is simple by \cite[Corollary 1.2.8]{HL10}. So, $\dim \operatorname{Ext}^0 (\cE, \cE) =1$.
		
		Now, recall that $\operatorname{Ext}^i (\cE, \cE) \cong \H{i} (\cE^{\vee} \otimes \cE)$. Then, let us tensor the sequence \eqref{ses:constructed_bundle_by_Serre} by $\cE^{\vee} = \cE((1-e)f)$ and consider the cohomology sequence of it. We have $\h{i} (\cE^{\vee} , \cE) = \h{i} (\cE \otimes \cI_{Z} (\xi +(1-e)f))$ for $i=2,3$. Because, $\h{i} (\cE (-\xi)) \leq \h{i} (\cO_{X_e}(-2 \xi + (e-1)f)) + \h{i}(\cI_{Z})$ and both $\h{i} (\cO_{X_e}(-2 \xi + (e-1)f))$ and $\h{i}(\cI_{Z})$ vanish for $i = 2,3$.
		
		Then, consider the cohomology of the sequence \eqref{ses:ideal_sheaf} after tensoring by $\cE (\xi + (1-e)f)$. Since $\cE = \cF(-\xi + (e-1)f)$ and $\mathcal{N}_{Z\vert X_e} \simeq \cF$, $\cE(\xi +(1-e)f)\vert_{Z} \simeq \cF \vert_{Z} = \oplus_{1}^{\alpha +1} (\cO_{\xi f} (1) + \cO_{\xi f} (1) (e))$. Since the latter has vanishing cohomology other than the level $0$, we have $\h{i} (\cE \otimes \cI_{Z} (\xi +(1-e)f)) = \h{i} (\cE (\xi + (1-e)f))$ for $i=2,3$.
		
		Now, consider the cohomology of the sequence \eqref{ses:constructed_bundle_by_Serre} after tensoring by $\cO_{X_e}(\xi +(1-e)f)$. Since $\h{i} (\cO_{X_e}) =0$ for $i >0$, we have $\h{i} (\cE (\xi + (1-e)f)) =  \cI_{Z} (2 \xi +(1-e)f)$ for $i = 2,3$.
		
		Then, consider the cohomology of the sequence \eqref{ses:ideal_sheaf} after tensoring by $\cO_{X_e} (2 \xi +(1-e)f)$. Since $\xi f \cdot (2 \xi +(1-e)f) = e+1 >0, \h{i} (\cO_{Z}(2 \xi +(e-1)f))=0$ for $i >0$. So, $\h{i} \cI_{Z} (2 \xi +(1-e)f) = \h{i} (\cO_{X_e} (2\xi +(1-e))f)$ for $i = 2,3$. We have $\h3(\cO_{X_e} (2\xi +(1-e))f) =0$ and 
		$$\h2 (\cO_{X_e} (2\xi +(1-e))f) = \begin{cases}
			0& \text{if } e \leq 3 \\
			\binom{e-2}{2}& \text{if } e\geq 4
 		\end{cases}$$
 		by Lemma \ref{lem0}.
 		
 		Hence, $\h3 (\cE^{\vee} \otimes \cE) =0$ and $\h2 (\cE^{\vee} \otimes \cE)  = \begin{cases}
 			0& \text{if } e \leq 3 \\
 			\binom{e-2}{2}& \text{if } e\geq 4
 		\end{cases}$.
 		
 		It remains to compute $\dim \operatorname{Ext}^1 (\cE, \cE) - \dim \operatorname{Ext}^2 (\cE, \cE)$. First notice that $c_i (\cE \otimes \cE^{\vee}) =0$ for $i=1,3$ and $c_2 (\cE \otimes \cE^{\vee}) = 4c_2(\cE) -c_1^2 (\cE)$. Then, we can apply Grothendieck–Riemann–Roch theorem for $\cE \otimes \cE^{\vee}$ and we will have
 		$$ \chi (\cE \otimes \cE^{\vee}) = \frac{1}{6}c_1(T_{X_e})c_2(T_{X_e}) - \frac{1}{2}c_1(T_{X_e})(4c_2 - c_1^2).
 		$$
 		Since we already have showed that $\dim \operatorname{Ext}^1 (\cE, \cE) =1$ and $\dim \operatorname{Ext}^3 (\cE, \cE) =0$, we have $\dim \operatorname{Ext}^1 (\cE, \cE) - \dim \operatorname{Ext}^2 (\cE, \cE) = (6+2e) \alpha +4\beta -(e-1)^2 -3$.
 	\end{proof}
    \begin{remark}
        The $H$-instanton bundle on $(X_e , H)$ constructed in Theorem \ref{thm:existence_serre} is earnest when $e \leq 3$. Because, one can see that 
        $$ \h2 (\cE (-(e+1)f)) = \begin{cases}
            0& \text{if } e \leq 3\\
            \binom{e-2}{2}& \text{if } e \geq 4
        \end{cases}
        $$
        by a similar cohomology computation in the proof Theorem \ref{thm:ext_computation}.
    \end{remark}

     \begin{remark}
        For $e=0$ and $e=2$, the construction above coincides with that in \cite{AM} section 5 and \cite{CG} section 5, respectively.
For $e=1$, however, the construction is slightly different from that in \cite{CCGM}. There they were considered sections of $\cE(f)$ instead of sections of $\cE(\xi)$. In \cite{AMMP}, the following definition is given: let $D$ be an eﬀective divisor, an instanton bundle $\cE$ is a $D$-’t Hooft bundle if and only if
$h^0(\cE(D))\not=0$. In the above construction with $e=1$ and in \cite{CCGM} are considered $\xi$-’t Hooft bundles and $f$-’t Hooft bundles respectively. Among
them, ‘t Hooft instantons for which the zero locus of a minimal global section lies on a hyperplane surface
section of F will be called special ‘t Hooft instantons. In this case the special ’t Hooft bundles are precisely those who are both $\xi$-’t Hooft  and $f$-’t Hooft. In \cite{AMMP} was studied the Del Pezzo threefold of degree 6 $F(0,1,2)$ which has the same hyperplane sections of $\PP(\cO_{\PP^2}\oplus\cO_{\PP^2}(1))$. Therefore, in our case too, we have the following result (see \cite{AMMP} Lemma 7.10):
for $k\geq 2$ there exist $\xi$-’t Hooft bundles which are not $f$-’t Hooft. For the case $k=1$ see \cite{CMP}.
    \end{remark}

    \section{H-instanton bundles on $X_e$ for $e \leq 3$}
	
	Let \( L \) be a curve in the class \( f^2 \) within \( A^2(X_e) \). Since \( L \) is a complete intersection, its structure sheaf has the following resolution:
	\begin{equation}\label{ses: resolution of L}
		0 \rightarrow \mathcal{O}_{X_e} (-2 f) \rightarrow \mathcal{O}_{X_e}(-f)^{\oplus 2} \rightarrow \mathcal{O}_{X_e} \rightarrow \mathcal{O}_L \rightarrow 0.
	\end{equation}
	In particular, we have $\mathcal{N}_{L|X_e} \simeq \mathcal{O}_L^{\oplus 2}$, so the determinant is \(\det(\mathcal{N}_{L|X_e}) = \mathcal{O}_L\).
	
	Moreover, the intersection number satisfies \( f^2 \cdot (\xi + f) =  1\), and we compute
	\[
	\chi(\mathcal{O}_L) = \chi(\mathcal{O}_{X_e}) - 2 \chi(\mathcal{O}_{X_e} (-f)) + \chi(\mathcal{O}_{X_e} (-2f)) = 1,
	\]
	since \(\chi(\mathcal{O}_{X_e} (-f)) = \chi(\mathcal{O}_{X_e} (-2 f)) = 0\) by Lemma \ref{lem0}. Therefore, \( L \) is a line in \( X_e \).
	
	From now on, we denote by \(\Lambda_L\) the component of the Hilbert scheme \({\rm Hilb}^{t+1}(X_e)\) containing these complete intersection curves. Since
	\[
	h^0(\mathcal{N}_{L|X_e}) = 2 \quad \text{and} \quad h^1(\mathcal{N}_{L|X_e}) = 0,
	\]
	we conclude that \(\Lambda_L\) is smooth and has dimension \( 2 \). In fact, $f^2$ is a pull-back of a point in $\mathbb{P}^2$ via projectivization map; that is, $\Lambda_L \simeq (\mathbb{P}^2)^{\vee}$.
	
	\begin{proposition}\label{prop: splitting type of existence_serre}
		Let $\cE$ be a bundle constructed in Theorem \ref{thm:existence_serre}. Then the restriction of $\cE$ to a generic line $L \in |f^2|$ splits as $\cO_{L}^{\oplus2} $.
	\end{proposition}
	
	\begin{proof}
		We know that $\cE$ fits into the short exact sequence in display \eqref{ses:constructed_bundle_by_Serre}. If we restrict $\cE$ to a generic line $L$ that can be chosen as not intersecting $Z$, then we will have 
		\begin{equation}\label{ses:splitting}
			0 \longrightarrow \cO_{L} (-1) \longrightarrow \cE \otimes \cO_{L} \longrightarrow \cO_{L}(1) \longrightarrow 0.
		\end{equation}
		 Since $\cE|_L$ split, we have $\cE \otimes \cO_{L}= \cO_{L}(s) \oplus \cO_{L}(-s)$ for some $s$. The cohomology of the sequence \eqref{ses:splitting} yields $\h0 (\cE \otimes \cO_{L}) = 2$. So, $s=0$ or $s=1$. But, we obtain the trivial extension in the case of $s=1$. So, $s=0$.
	\end{proof}
	
	\begin{proposition} \label{prop: deformation}
		Let $\cE$ be an $\mu$-stable $H$-instanton bundle on $X_e$ with $c_2(\cE) = \alpha \xi f + \beta f^2$ and $\mathrm{Ext}^2 (\cE, \cE) = \mathrm{Ext}^3 (\cE, \cE) =0$ when $e \leq 3$. Let $L$ be a generic line of class $|f^2|$ such that $\cE \otimes \cO_{L} = \cO_{L}^2$ and $\cE_{\varphi}$ be the kernel sheaf of the general map $\varphi: \cE \rightarrow \cO_L$. Then
		\begin{enumerate}
			\item $c_1 (\cE_{\varphi}) =0$, $c_2 (\cE_{\varphi}) = \alpha \xi f + (\beta +1) f^2$ and $c_3 (\cE_{\varphi}) = 0$.
			\item $\cE_{\varphi}$ is $\mu$-stable.
			\item $\h0 (\cE_{\varphi}) = \h1 (\cE_{\varphi}(-H)) =0$.
			\item $\cE_{\varphi} \otimes \cO_{L^{'}} = \cO_{L^{'}}^2$ for a generic $L^{'} \in |f^2|$.
			\item $\dim \mathrm{Ext}^1 (\cE_{\varphi}, \cE_{\varphi}) = \dim \mathrm{Ext}^1 (\cE, \cE) +4$ and $\mathrm{Ext}^2 (\cE_{\varphi}, \cE_{\varphi}) = \mathrm{Ext}^3 (\cE_{\varphi}, \cE_{\varphi}) =0$.
		\end{enumerate}
	\end{proposition}
	
	\begin{proof}
		We have the short exact sequence
		\begin{equation}\label{ses:kernel sheaf}
			0 \rightarrow \cE_{\varphi} \rightarrow \cE \rightarrow \cO_{L} \rightarrow 0.
		\end{equation}
		Since $c_1 (\cO_{L}) = 0$ and $c_2(\cO_{L}) = -f^2$, we have $c_1 (\cE_{\varphi}) =0$ and $c_2 (\cE_{\varphi}) = \alpha \xi f + (\beta +1) f^2$ thanks to the sequence \eqref{ses:kernel sheaf}. Also, $c_3 (\cE_{\varphi}) = 0$ by a straightforward computation with Chern polynomials.
		
		Notice that each sheaf destabilizing $\cE_{\varphi}$ also destabilize $\cE$, $\cE_{\varphi}$ and $\cE$ have same slope. So, $\cE_{\varphi}$ is also $\mu$-stable.
		
		The cohomology of the sequence  \eqref{ses:kernel sheaf} gives us $\h0(\cE_{\varphi}) \leq \h0 (\cE) =0$; i.e, $\h0(\cE_{\varphi}) =0$. Similarly, $\h1 (\cE_{\varphi} (-H)) \leq \h1(\cE (-H)) + \h0 (\cO_{L} \otimes \cO_{X_e}(-H))$. Since $\cE$ is an $H$-instanton and $\h0 (\cO_{L} \otimes \cO_{X_e}(-H)) \simeq \h0 (\cO_{\PP^1} (-1)) =0$, we have $\h1 (\cE_{\varphi} (-H)) =0$.
		
		Since a generic element of $|f^2|$ is a pull-back of a point in $\pp^2$ via projectivization map, we can choose two generic element of $|f^2|$ disjoint from each other. So, the restriction of sequence \eqref{ses:kernel sheaf} to each general $L^{'} \in |f^2|$ remains exact; that is, $\cE_{\varphi} \otimes \cO_{L^{'}} \simeq \cE \otimes \cO_{L^{'}} = \cO_{L^{'}}^2$.
		
		Now, let us apply $\Hom_{X_e}\big(\cE,-\big)$ to Sequence \eqref{ses:kernel sheaf}. Then, we obtain that $\mathrm{Ext}^j_{X_e}\big(\cE,\cE_\varphi\big)=0$ for $j=2,3$, because $\mathrm{Ext}^j (\cE, \cE) =0$ for $j=2,3$ and $\mathrm{Ext}^i_{X_e}\big(\cE,\cO_L \big) \cong H^i\big(L,\cO_L^{\oplus2}\big) =0$, $i\ge1$. Thus, by applying $\Hom_{X_e} \big(-,\cE_\varphi\big)$ to Sequence \eqref{ses:kernel sheaf} we obtain 
		$$
		\mathrm{Ext}^2_{X_e}\big(\cE_\varphi,\cE_\varphi\big)\subseteq\mathrm{Ext}^3_{X_e}\big(\cO_L,\cE_\varphi\big) \text{ and }
		\mathrm{Ext}^3_{X_e}\big(\cE,\cE_\varphi \big) \twoheadrightarrow \mathrm{Ext}^3_{X_e}\big(\cE_\varphi,\cE_\varphi\big).
		$$
		Since $\mathrm{Ext}^3_{X_e}\big(\cE,\cE_\varphi \big) =0$, we have $\mathrm{Ext}^3_{X_e}\big(\cE_\varphi,\cE_\varphi\big) =0$. On the other hand, 
		Serre duality implies $\mathrm{Ext}^3_{X_e}\big(\cO_L,\cE_\varphi\big)\cong \Hom_{X_e}\big(\cE_\varphi,\cO_L \otimes K_{X_e}\big)^{\vee} $. Since $f^2 \cdot K_{X_e}= -2$, we have $\Hom_{X_e}\big(\cE_\varphi,\cO_L \otimes K_{X_e}\big) = \Hom_{X_e}\big(\cE_\varphi,\cO_L (-2)\big) $. Since $L$ is a complete intersection subscheme in $X_e$, we have the isomorphism $\mathcal{E}xt^i_{X_e}\big(\cO_L,\cO_L\big) \simeq \mathcal{E}xt^{i-1}_{X_e}\big( I_L, \cO_L\big) = \wedge^{i} \mathcal{N}_{L\vert X_e}$. From $\mathcal{N}_{L\vert X_e}\cong\cO_{\mathbb{P}^1}^{\oplus2}$ we deduce  $\mathcal{E}xt^1_{X_e}\big(\cO_L,\cO_L(-2)\big)\cong\cO_L(-2)^{\oplus2}$. By the above equalities, applying $\mathcal{H}om_{X_e}\big(-,\cO_L(-2)\big)$ to Sequence \eqref{ses:kernel sheaf} we obtain the exact sequence
		\begin{align*}
			0\longrightarrow\cO_L(-2)\longrightarrow\cO_L(-2)^{\oplus2}\longrightarrow\mathcal{H}om_{X_e}\big(\cE_\varphi,\cO_L(-2)\big)\longrightarrow\cO_L(-2)^{\oplus2}\longrightarrow0,
		\end{align*}
		because $\mathcal{E}xt^1_{X_e}\big(\cE,\cO_L(-2)\big)=0$. It follows that 
		$$
		\Hom_{X_e}\big(\cE_\varphi,\cO_L(-2)\big)=H^0\big(X_e,\mathcal{H}om_{X_e}\big(\cE_\varphi,\cO_L(-2)\big)\big)=0,
		$$
		hence $\mathrm{Ext}^2_{X_e}\big(\cE_\varphi,\cE_\varphi\big)=0$. Lastly, $\mathrm{Ext}^0_{X_e}\big(\cE_\varphi,\cE_\varphi\big)=1$ thanks to the $\mu$-stability of $\cE_\varphi$, which implies the simpleness of $\cE_\varphi$. It follows that
		$$
		\dim\mathrm{Ext}^1_{X_e}\big(\cE_\varphi,\cE_\varphi\big)=1-\chi(\cE_\varphi,\cE_\varphi).
		$$
		
		By applying $\Hom_{X_e}(\cE_\varphi,-\big)$, $\Hom_{X_e}(-,\cE\big)$ and $\Hom_{X_e}(-,\cO_L\big)$ to Sequence \eqref{ses:kernel sheaf} we deduce that
		$$
		\chi(\cE_\varphi,\cE_\varphi)=\chi(\cE,\cE)-\chi(\cE,\cO_L)-\chi(\cO_L,\cE)+\chi(\cO_L,\cO_L).
		$$
		Also we know that
		$$
		\chi(\cE,\cE)=1-\dim\mathrm{Ext}^1_{X_e}\big(\cE,\cE\big).
		$$
		We have $\chi(\cE,\cO_L)=2\chi(\cO_L)=2$ and $\chi(\cO_L,\cE)=-2\chi(\cO_L(-2))=2$. Finally, we deduce from the cohomology of Sequence \eqref{ses: resolution of L} that $\chi(\cO_L,\cO_L)=0$.
		Thus $\dim\mathrm{Ext}^1_{X_e}\big(\cE_\varphi,\cE_\varphi\big)= \dim\mathrm{Ext}^1_{X_e}\big(\cE,\cE\big) +4$. 
	\end{proof}

    \begin{remark}
		\label{rUnique}
		Let $\cE_\varphi$ be the sheaf defined by Sequence \eqref{ses:kernel sheaf}. If $\Gamma\subseteq X_e$ is  any subscheme of pure dimension $1$, then
		$$
		\mathrm{Ext}^1_{X_e} \big(\cO_\Gamma,\cE_\varphi\big)=\left\lbrace\begin{array}{ll} 
			0\quad&\text{if $L \not\subseteq \Gamma$,}\\
			1\quad&\text{if $L \subseteq \Gamma$.}
		\end{array}\right.
		$$
		Indeed, by applying $\Hom_{X_e} \big(\cO_\Gamma,-\big)$ to Sequence \eqref{ses:kernel sheaf} and taking into account Serre Duality we obtain 
		$$
		\mathrm{Ext}^1_{X_e} \big(\cO_\Gamma,\cE_\varphi\big)\cong \Hom_{X_e} \big(\cO_\Gamma,\cO_L \big)\cong H^0\big(X_e , \sHom_{X_e} \big(\cO_\Gamma,\cO_L\big)\big).
		$$
		We conclude by noticing that
		$$
		\sHom_{X_e} \big(\cO_\Gamma,\cO_L \big)=\left\lbrace\begin{array}{ll} 
			0\quad&\text{if $L\not\subseteq \Gamma$,}\\
			\cO_L \quad&\text{if $L \subseteq \Gamma$.}
		\end{array}\right.
		$$
	\end{remark}
	
	\begin{remark}
		\label{rBiDual}
		The bundle $\cE$ is $\mu$--stable, hence simple. By applying $\Hom_{X_e} \big(-,\cE\big)$ to Sequence \eqref{ses:kernel sheaf} and taking into account Serre Duality we obtain $\Hom_{X_e} \big(\cE_\varphi,\cE\big)\cong \mathbb{C}$. 
		
		By applying $\sHom_{X_e} \big(-,\cO_{X_e} \big)$ to Sequence \eqref{ses:kernel sheaf} we deduce $\cE_\varphi^{\vee\vee}\cong\cE$, because $\sExt^1_{X_e} \big(\cO_L,\cO_{X_e} \big)\cong \sExt^1_{X_e} \big(\cO_L ,\omega_{X_e} \big)\otimes \cO_{X_e} (2\xi + (3-e)f)=0$ by applying the same argument in \cite[Lemma 7.3]{Har}. 
		
		The above isomorphisms show that the inclusion $\cE_\varphi\to\cE$ can be uniquely identified with the canonical monomorphism $\cE_\varphi\to\cE_\varphi^{\vee\vee}$, i.e. Sequence \eqref{ses:kernel sheaf} is canonically isomorphic with 
		\begin{equation*}
			0\longrightarrow\cE_\varphi\longrightarrow\cE_\varphi^{\vee\vee}\longrightarrow\cO_L \longrightarrow0.
		\end{equation*}
	\end{remark}
	
    \begin{theorem}
		For each $\alpha, \beta \in \mathbb{Z}$ with $\alpha > e$ and $\beta \geq 0$, there exists an earnest $\mu$-stable $H$-instanton bundle $\cE$ with charge $c_2(\cE) = \alpha \xi f + \beta f^2$ on $X_e$ where $e \leq 3$ such that 
		\[\operatorname{dim} \operatorname{Ext}^1 (\cE,\cE) = (2e+6)\alpha +4\beta -(e-1)^2 -3, \ \ \ \operatorname{Ext}^2 (\cE, \cE)= \operatorname{Ext}^3(\cE,\cE)=0.
		\]
	\end{theorem}
	
	\begin{proof}
        We proceed by induction to establish the existence of $H$-instanton bundles on $X_e$ with charge $\alpha \xi f+\beta f^2$ under suitable assumptions on $\alpha,\beta\in \mathbb{Z}$. The initial step of the induction, corresponding to $\beta=0$ has already been treated in Theorem \ref{thm:existence_serre}, while the associated $\operatorname{Ext}^i$-group calculations appear in Theorem \ref{thm:ext_computation}.

        Assume therefore that there exists an earnest $\mu$--stable $H$-instanton bundle $\cE$ with charge $\alpha \xi f+\beta f^2$, where $\alpha \geq e$ and $\beta \geq 0$, satisfying
		$$
		\operatorname{dim} \operatorname{Ext}^1 (\cE,\cE) =(2e+6)\alpha +4\beta -(e-1)^2 -3 ,\qquad  \operatorname{Ext}^2 (\cE,\cE) = \operatorname{Ext}^3 (\cE,\cE)=0.
		$$ 
        Consequently, $\cE$ determines a point of a component $\cM$ of the moduli space of $\mu$--stable rank $2$ sheaves with Chern classes $c_i(\cE)$. This component is smooth at $\cE$ and has dimension
		$$
		\operatorname{dim} \operatorname{Ext}^1 (\cE,\cE) =(2e+6)\alpha +4\beta -(e-1)^2 -3.
		$$
		In Proposition \ref{prop: deformation}, we constructed a sheaf $\cE_\varphi$ that satisfies all the defining properties of an $H$-instanton bundle except locally freeness. This sheaf corresponds to a point in a component $\cM_1$ of the moduli space of $\mu$--stable rank $2$ sheaves with Chern classes $c_i(\cE_\varphi)$. Moreover, $\cM_1$ is smooth at $\cE_\varphi$ with dimension
		$$
		\operatorname{dim} \operatorname{Ext}^1 (\cE,\cE) =(2e+6)\alpha +4\beta -(e-1)^2 +1.
		$$
		
		The subset $\cM_{bad}\subseteq\cM_1$ consisting of such sheaves is parametrized by the choice of $\cE \in \cM$, by a line $L \in \Lambda_L$ and by a point of $\mathbb{P} \big( \mathrm{Hom}_{\mathbb{P}^1} \big(\cO_{\mathbb{P}^1}^{\oplus2}, \cO_{\mathbb{P}^1} \big) \big)$ corresponding to the morphism $\mathrm{Hom}_{X_e} \big( \cE \otimes \cO_L, \cO_L \big)$ up to scalars. It follows that
		$$
		\dim(\cM_{bad})\le (2e+6)\alpha +4\beta -(e-1)^2  < (2e+6)\alpha +4\beta -(e-1)^2 +1 = \dim(\cM_1).
		$$
		Therefore, there exists a flat family $\frak E \to S$ over an integral base such that $\frak E_{s_0}\cong\cE_\varphi$ and $\frak E_s \not\in \cM_1$ for $s \ne s_0$. By \cite[Satz 3]{B--P--S}, we may further assume that $\operatorname{Ext}^2_{X_e}\big(\frak E_s,\frak E_s\big)=0$ for $s\in S$.
		
		For each $s\in S$, there is a canonical exact sequence
		\begin{equation}
			\label{seqBidual}
			0\longrightarrow\frak E_s\longrightarrow(\frak E_s)^{\vee\vee}\longrightarrow\frak T_s\longrightarrow0.
		\end{equation}
		where $\frak T_s$ is a torsion sheaf on $X_e$. Since $(\frak E_s)^{\vee\vee}$ is reflexive, the locus where it fails to be locally free has degree $c_3(\frak E_s)\ge0$, by \cite[Proposition 2.6]{Ha1}.
		
		Because $\frak E_s$ is torsion--free, the support $\Gamma_s$ of $\frak T_s$ has codimension at least $2$ (see \cite[Corollary to Theorem II.1.1.8]{O--S--S}). By semicontinuity we may assume that for general $s \in S$, the sheaf $\frak E_s$ is a $\mu$--stable, satisfies $h^0\big(X_e, \frak E_s \big)=h^1\big(X_e ,\frak E_s(-H)\big)=0$, and restricts as $\frak E_s\otimes\cO_L\cong\cO_{\mathbb{P}^1}^{\oplus2}$ for a general $L \in \Lambda_L$. It remains to verify that $\frak E_s$ is locally free.
		
		Suppose that $h^0\big(X_e , (\frak E_s)^{\vee\vee}(-H)\big)\ne0$. Then there exists a subsheaf $\mathcal L \subseteq (\frak E_s)^{\vee\vee}$ with $\mu(\mathcal L)= H^3>0$. Letting $\mathcal{K}:= \mathcal L\cap\frak E_s$, we have $\mu(\mathcal{K})=\mu(\mathcal L)$, since the cokernel of the map $\mathcal{K} \to \mathcal L$ induced by $\frak E_s \to (\frak E_s)^{\vee \vee}$ has support contained in $\Gamma_s$. This  contradicts the $\mu$--semistability of $\frak E_s$, as $\mathcal{K} \subseteq\frak E_s$. From the cohomology of sequence \eqref{seqBidual} tensored by $\cO_{X_e}(-H)$, we deduce
		\begin{equation}
			\label{VanishingT}
			h^0\big(\Gamma_s,\frak T_s(-H)\big)=0.
		\end{equation}
		Hence $\Gamma_s$ has no embedded points and is therefore pure of dimension $1$. Two cases may occur. 
		
		In the first case, there exists $s\in S$ such that $\Gamma_s=\emptyset$. Then $\frak E_s\cong(\frak E_s)^{\vee\vee}$ is reflexive and since $c_3(\frak E_s)=0$, $\frak E_s$ is a vector bundle. We already know that $\operatorname{Ext}^2_{X_e} \big( \frak E_s,\frak E_s \big)=0$, and a straightforward computation yields the remaining dimensions, completing the proof.
		
		In the second case, $\Gamma_s\ne\emptyset$ for each $s\in S$. Let $H'$ be a general hyperplane section not intersecting the $0$--dimensional locus of points where $(\frak E_s)^{\vee\vee}$ is not locally free. Thus, the restriction of Sequence \eqref{seqBidual} to $H'$ remains exact. Since $H'$ is general and $(\frak E_s)^{\vee\vee}\otimes\cO_{H'}$ is reflexive, it is locally free on $H'$. In particular
		$$
		h^i\big(H',(\frak E_s)^{\vee\vee} \otimes \cO_{H'} (\lambda H) \big)=h^i \big( H' ,\cE \otimes \cO_{H'} (\lambda H)\big)=0
		$$
		for $i\le1$ and $\lambda \ll0$. By semicontinuity, this implies
		\begin{align*}
			h^0 \big( \Gamma_s \cap H' ,\frak T_s \otimes \cO_{H'} \big)&=h^1\big( H' ,\frak E_s \otimes \cO_{H'} (\lambda H) \big) \le \\
			&\le h^1\big(H', \cE_\varphi \otimes \cO_{H'} (\lambda H)\big)=h^0\big(L\cap H',\cO_L \otimes \cO_{H'} \big)=1.
		\end{align*}
		Thus $\Gamma_s$ must be a line. Consequently, $\frak T_s$ decomposes as the direct sum of a locally free sheaf and a torsion sheaf. Equality \eqref{VanishingT} further implies that $\frak T_s$ is  invertible along $\Gamma_s$.
		
		We now compute the class of $\Gamma_s$ in $A^2(X_e)$. For any line bundle $\cO_{X_e} (D)\in\Pic(X_e)$, $c_3((\frak E_s)^{\vee\vee}(D))$, $c_3(\frak E_s(D))$ and $c_2(\frak T_s(D))$ are independent of $D$, by \cite[Lemma 2.1]{Ha1}.  A direct Chern class computation from Sequence \eqref{seqBidual} then gives 
		$$
		c_3(\frak T_s(D))=c_3((\frak E_s)^{\vee\vee})-2Dc_2(\frak T_s).
		$$
		Then Equality \eqref{RRgeneral} for the sheaf $\frak T(D)$ yields
		\begin{equation*}
			\label{RRT}
			\chi(\frak T_s(D))=\frac {c_3((\frak E_s)^{\vee\vee})} 2-(h+D)c_2(\frak T_s).
		\end{equation*}
		Taking $D=-H$ and applying Equality \eqref{VanishingT}, we find
		$$
		0\le \frac {c_3((\frak E_s)^{\vee\vee})} 2=\chi(\frak T_s(D))=-h^1\big(\Gamma_s,\frak T_s(-H)\big)\le0,
		$$
		hence $c_3((\frak E_s)^{\vee\vee})=0$ for all $s\in S$. By \cite[Proposition 2.6]{Ha1}, $(\frak E_s)^{\vee\vee}$ is a vector bundle for each $s\in S$  and $c_3(\frak T_s)=0$.
		
		Write $c_2(\frak T_s)=\eta_s\xi f+\vartheta_s f^2$ and $\cO_{X_e} (D):=\cO_{X_e} (a\xi+b f)$. For $a,b \ll0$, we have $h^1\big(X_e,(\frak E_s)^{\vee\vee}(D)\big)=0$, so Equality \eqref{VanishingT} yields
		\begin{equation}
			\label{A}
			\begin{aligned}
				(1+a)(e \eta_s+\vartheta_s)+(1+b )\vartheta_s&=h^1\big(\Gamma_s,\frak T_s(D)\big)=\\
				&=h^2\big(X_e, \frak E_s(D)\big)\le h^2\big(X_e,\cE_\varphi(D)\big),
			\end{aligned}
		\end{equation}
		where the inequality follows from semicontinuity. For $a,b \ll0$, we also have $h^1\big(X_e,\cE(D)\big)=h^2\big(X_e,\cE(D)\big)=0$, so
		\begin{equation}
			\label{B}
			h^2\big(X_e, \cE_\varphi(D)\big)=h^1\big(L, \cO_L (D) \big)=-1-a. 
		\end{equation}
		Combining Inequality \eqref{A} and Equality \eqref{B} gives
		\begin{equation}
			\label{C}
			(1+a)(e \eta_s+\vartheta_s+1)+(1+b )\eta_s\le0,
		\end{equation}
		for $a,b \ll0$.
		
		As observed in \cite[Example 15.3.1]{Fu},
		$$
		(e+1)\eta_s+\vartheta_s=\deg(c_2(\frak T_s))=-\rk(\frak T_s)\deg(\Gamma_s)=-1. 
		$$
		Substituting into Equality \eqref{C} yields $(b -a)\eta_s<0$ for all $a,b \ll0$, forcing $\eta_s=0$ and hence $\vartheta_s=-1$. Thus $c_2(\frak T_s)=-f^2$ and $\Gamma_s=L_s\in\Lambda_L$. Since $c_3(\frak T_s)=0$, we obtain $\frak T_s\cong\cO_{L_s}$. This determines a morphism $S\to \Lambda_L$, so $\frak T\to S$ is a flat family. 
		
		The exact sequence
		\begin{equation}
			\label{seqDeformFamily}
			0\longrightarrow\frak E\longrightarrow\frak E^{\vee\vee}\longrightarrow\frak T\longrightarrow0
		\end{equation}
		together with the flatness of families $\frak E\to S$ and $\frak T\to S$ implies the flatness of the induced family $\frak E^{\vee\vee}\to S$. In particular $(\frak E^{\vee\vee})_s\cong(\frak E_s)^{\vee\vee}$ for all $s\in S$, and each fiber $(\frak E^{\vee\vee})_s$ is  a vector bundle. Restricting Sequence \eqref{seqDeformFamily} to $F\times\{\ s_0\ \}\subseteq F\times S$ yields a non--split sequence, which must coincide with Sequence \eqref{ses:kernel sheaf} by Remark \ref{rUnique}. Hence $(\frak E_{s_0})^{\vee\vee}\cong\cE$, so $(\frak E_{s})^{\vee\vee}\in\cM$ and therefore $\frak E_s\in\cM_{bad}$ for all $s\in S$, contradicting our initial assumption. We conclude that the case $\Gamma_s\ne\emptyset$ for all $s\in S$ cannot occur.

        Lastly, the property of being earnest follows from  semicontinuity and Theorem \ref{thm:earnest}.
	\end{proof}
\bibliographystyle{amsplain}

\bigskip
\noindent
Ozhan Genc,\\
Faculty of Mathematics and Computer Science, Jagiellonian University,\\
ul. prof. Stanisława Łojasiewicza 6,\\
30-348 Kraków, Poland\\
e-mail: {\tt ozhangenc@gmail.com}

\bigskip
\noindent
Francesco Malaspina,\\
Dipartimento di Scienze Matematiche, Politecnico di Torino,\\
c.so Duca degli Abruzzi 24,\\
10129 Torino, Italy\\
e-mail: {\tt francesco.malaspina@polito.it}

\end{document}